\documentclass[11pt,reqno]{amsart}

\usepackage{amsmath,amssymb,amsthm, comment,graphicx,cite}
\usepackage{tikz}
\usepackage{fancyhdr}
\usepackage{epsfig}
\usepackage{mathrsfs}

\newtheorem{theorem}{Theorem}[section]

\newtheorem{lemma}{Lemma}[section]
\newtheorem{remark}{Remark}[section]
\newtheorem{proposition}{Proposition}[section]

\numberwithin{equation}{section}
\numberwithin{figure}{section}

\begin{document}

\title[Stability of supersonic contact discontinuity in a nozzle]
{Stability of supersonic contact discontinuity for two-dimensional
steady compressible Euler flows in a finite nozzle}

\author{Feimin Huang }
\address{Institute of Applied Mathematics,
Academy of Mathematics and Systems Science,
Chinese Academy of Sciences, Beijing 100190,
 China}
\email{\tt fhuang@amt.ac.cn}

\author{Jie Kuang }
\address{ Wuhan Institute of Physics and Mathematics,
Chinese Academy of Sciences, Wuhan 430071,
 China; and
Institute of Applied Mathematics,
Academy of Mathematics and Systems Science,
Chinese Academy of Sciences, Beijing 100190,
 China}
\email{\tt jkuang@wipm.ac.cn,  \ \  jkuang@amss.ac.cn}

\author{Dehua Wang}
\address{Department of Mathematics,
University of Pittsburgh,
Pittsburgh, PA 15260, USA.}
\email{\tt dwang@math.pitt.edu}

\author{Wei Xiang}
\address{Department of Mathematics,
City University of Hong Kong,
Kowloon, Hong Kong,   China}
\email{\tt  weixiang@cityu.edu.hk}

\keywords{Contact discontinuity, supersonic flow, free boundary, compressible Euler equation, finitely long nozzle.}
\subjclass[2010]{35B07, 35B20, 35D30; 76J20, 76L99, 76N10}
\date{}

\begin{abstract}
In this paper, we study the stability of supersonic contact discontinuity
for the two-dimensional steady compressible Euler flows in a finitely long nozzle of varying
cross-sections. We formulate the problem as an initial-boundary value problem with the contact
discontinuity as a free boundary. To deal with the free boundary value problem,
we employ the Lagrangian transformation to straighten the contact discontinuity
and then the free boundary value problem becomes a fixed boundary value problem.
We develop  an iteration scheme and establish some novel estimates of solutions for the first order of
hyperbolic equations on a cornered domain. Finally, by using the inverse
Lagrangian transformation and under the assumption that the incoming flows
and the nozzle walls are smooth perturbations of the background state,
we prove that the original free boundary problem admits a unique weak solution which
is a small perturbation of the background state and the solution consists of two smooth supersonic
flows separated by a smooth contact discontinuity.
\end{abstract}

\maketitle
\section{Introduction }\setcounter{equation}{0}
The two-dimensional steady  full Euler system of compressible flows is of the following form:
\begin{eqnarray}\label{eq:1.1}
\left\{
\begin{array}{llll}
     \partial_x(\rho u)+\partial_y(\rho v)=0, \\
     \partial_{x}(\rho u^{2}+p)+\partial_{y}(\rho uv)=0, \\
   \partial_{x}(\rho uv)+\partial_{y}(\rho v^{2}+p)=0, \\
     \partial_{x}\big((\rho E+p)u\big)+\partial_{y}\big((\rho E+p)v\big)=0,
     \end{array}
     \right.
\end{eqnarray}
where $(u,v),\ p,\ \rho$  stand for the velocity, pressure, and density, respectively, and
\begin{equation*}
E=\frac{1}{2}(u^2+v^2)+\frac{p}{(\gamma-1)\rho} 
\end{equation*}
is the energy with adiabatic exponent $\gamma>1$. Here, the pressure $p$ and density $\rho$
satisfy the   constitutive relation:
$p=A(S)\rho^{\gamma}$, 
with $S$  the entropy. %
If the solution of \eqref{eq:1.1} is classical, \emph{i.e.}, $U=(u, v, p, \rho)\in \mathcal{C}^{1}$,
 we have the following transport equations:
\begin{eqnarray*}  
(u,v)\cdot\nabla B=0,\quad\quad  (u,v)\cdot\nabla \big(\frac{p}{\rho^{\gamma}}\big)=0,
\end{eqnarray*}
where
\begin{eqnarray*}   
 B= \frac{1}{2}(u^2+v^2)+\frac{\gamma p}{(\gamma-1)\rho}.
\end{eqnarray*}
This means that the entropy is preserved and the Bernoulli law holds along each streamline. 
We denote the sonic speed of the flow by $c=\sqrt{\frac{\gamma p}{\rho}}$.
The flow is called supersonic if $u^{2}+v^{2}>c^{2}$,  subsonic if $u^{2}+v^{2}<c^{2}$, and sonic if
 $u^{2}+v^{2}=c^{2}$.
For the supersonic flow the system \eqref{eq:1.1} is hyperbolic, while for the subsonic flow
the system \eqref{eq:1.1}  is elliptic, and in general the system \eqref{eq:1.1}  is of  hyperbolic-elliptic mixed type.

\par In this paper, we are concerned with the stability of supersonic contact discontinuity
governed by the two-dimensional steady full Euler equations in a finite nozzle (see Fig. 1.\ref{fig11}).
The domain in the nozzle can be described by
\begin{equation*}
\Omega:=\big\{(x,y)\in \mathbb{R}^{2}: g_{-}(x)<y<g_{+}(x),\ 0<x<L \big\}, 
\end{equation*}
where $g_{-},\ g_{+}$  are functions of $x$.
The lower and upper boundaries of the nozzle
are denoted by $\Gamma_{-}$ and $\Gamma_{+}$, \emph{i.e.},
\begin{eqnarray*}
\Gamma_{\pm}:=\big\{(x,y):y=g_{\pm}(x),\ 0<x<L \big\}.
\end{eqnarray*}

\vspace{10pt}
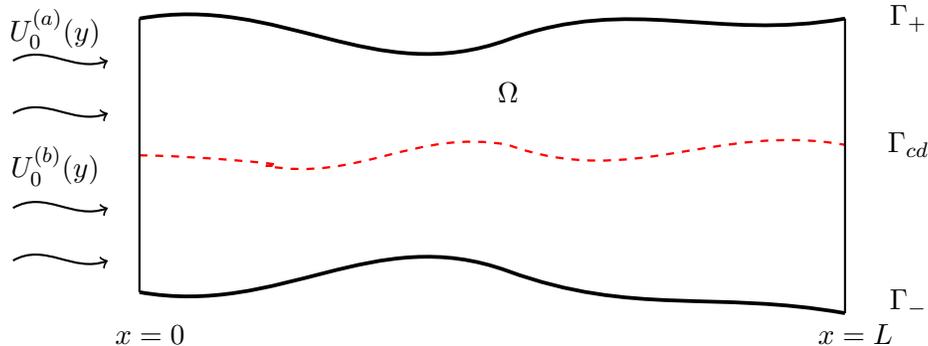
\begin{figure}[ht] \label{fig11}
\begin{center}
\begin{tikzpicture}[scale=1.4]
\draw [thick][->] (-5.7,1.9)to[out=30, in=-160](-4.8, 1.9);
\draw [thick][->] (-5.7,1.4)to[out=30, in=-160](-4.8, 1.4);
\draw [thick][->] (-5.7,0.5)to[out=30, in=-160](-4.8, 0.5);
\draw [thick][->] (-5.7,0.0)to[out=30, in=-160](-4.8, 0.0);
\draw [line width=0.05cm](-4.5, 2.3)to
[out=10,in=-160](-1,2.1)to [out=20,in=-170](2.2,2.3);
\draw [line width=0.05cm](-4.5, -0.3)to
[out=-10,in=160](-1,-0.1)to [out=-20,in=170](2.2,-0.5);
\draw [line width=0.03cm][dashed][red] (-4.5,1) to [out=0, in=0]
(-3.3, 0.9)to [out=-10,in=170](-1,1.1)to [out=-20,in=170](2.2,1.1);
\draw [thick] (-4.5,2.3)--(-4.5, -0.3);
\draw [thick] (2.2,2.3)--(2.2, -0.5);
\node at (2.8,2.3) {$\Gamma_{+}$};
\node at (2.8,-0.4) {$\Gamma_{-}$};
\node at (2.8, 1.1) {$\Gamma_{cd}$};
\node at (-1.0, 1.6) {$\Omega$};
\node at (-5.3, 2.2) {$U^{(a)}_{0}(y)$};
\node at (-5.3, 0.9) {$U^{(b)}_{0}(y)$};
\node at (-4.4, -0.7) {$x=0$};
\node at (2.3, -0.7) {$x=L$};
\end{tikzpicture}
\end{center}
\caption{Supersonic contact discontinuity in a finite  nozzle}
\end{figure}

There have been many studies in literature on the subsonic flows in  nozzles.
For the subsonic flow in infinitely or finitely long nozzles,   L. Bers in \cite{bl} first studied the existence
and uniqueness of subsonic irrotational flow in a two-dimensional infinitely long nozzle for a given appropriate
incoming mass flux. This argument was verified rigorously in \cite{xx1} for two-dimensional  nozzles,  in \cite{xx2} for
axially symmetric nozzles, and in \cite{dxy} for multi-dimensional nozzles.
The existence and uniqueness of isentropic flow  was obtained in a two-dimensional  infinitely long nozzle with small non-zero vorticity
 in \cite{xy3}, and in the axi-symmetric case with zero swirl  in \cite{dd}. For the full
Euler flow, similar results were obtained in \cite{cdx} and then \cite{dl, dxx}.
M. Bae in \cite{bm} showed the stability of straight contact discontinuities in a two-dimensional  almost
flat and infinity long nozzles. Recently, the sign condition and the small assumption were removed in \cite{chwx}
by applying the compensated compactness argument. The sonic-subsonic limit for subsonic flow in infinitely long
nozzles was also studied in \cite{cdsw,chw,hww,wx1,wx2}, while the incompressible limit was studied in \cite{chwx1}.
For the subsonic flow in a finitely long nozzle, there are only a few results. It is well-known that it
is ill-posed when the pressure is assigned  at both the inlet and the outlet. So how to find an appropriate boundary
condition is one of the major issues; see  \cite{dwx} for more details.
Another important problem for the steady Euler flows is the stability of the transonic shock in nozzles.
G.-Q. Chen and M. Feldman \cite{cf1,cf2} studied the multi-dimensional transonic shocks in a finite or infinitly long straight
nozzles for the potential flows,  and some further related results  were obtained in \cite{cf4, xy1}. The
transonic shock governed by the Euler equations was studied by  S.-X. Chen in \cite{csx1}, where he studied the two-dimensional Euler
flows with the uniform Bernoulli constant for the incoming flow with some symmetric structure and proved  the stability of
the transonic shock in a finitely long duct. Then the result was  generalized in \cite{y}   to
  finitely long nozzles as  small perturbations of a straight one. The uniqueness for the transonic
shocks in two-dimensional steady Euler flows in a finitely long duct was considered in \cite{fly} for a
class of piecewise $C^{1}$ smooth functions. For the three-dimensional case we   refer to \cite{csx2, cy, lxy, xyy, xy2,xy3}.
The transonic shock in finitely and infinitely long nozzles was studied in \cite{ccf,ccs}
for  the full Euler system. Moreover, there are also some works on the transonic shock in a divergent
or a De laval nozzle in \cite{bf,csx3,lxy1, lxy2} and references therein.

For the supersonic flow in nozzles there are only a few results. The reason is that it is very difficult
to control the solutions due to the reflection of the characteristics. The solutions are expected to blow up if the
nozzles are sufficiently long. Actually, we show that the non-constant irrotational flow will generally blow up in the semi-infinitely long flat nozzle  in the Appendix (see Theorem \ref{thm:6.1}).  The problem of the supersonic flow in nozzles is different from the problem of supersonic flow past a solid wall
(see in \cite{czz,qx}), or the local stability of straight contact discontinuity without boundary
(see in \cite{wy1,wy2}), for which the decay estimates are expected or there is no reflection on
the characteristics.  So far almost all the results on the supersonic flow in nozzles are about the nozzles with special structure, for example,
the expanding nozzles.  Chen and  Qu in \cite{cq} studied steady supersonic potential flow with rarefaction waves in
an expanding and infinitely long nozzles and found a sufficient condition to determine whether the vacuum state appears or not,
which is devoted to one of the questions proposed by  Courant and  Friedrichs in \cite{cf}. The smooth potential flow
without rarefaction waves was considered in \cite{wx3} for the two-dimensional case and in \cite{xy} for the multi-dimensional case.

To our best knowledge this paper is the first on the supersonic steady contact discontinuity in nozzles.
The problem of two-dimensional supersonic contact discontinuity in nozzles governed by the steady Euler equations
can be formulated as a nonlinear initial-boundary value problem for the first order quasilinear hyperbolic system, with
the contact discontinuity as a free boundary as well as the two nozzle walls as fixed boundaries.
For the study of the nonlinear stability of supersonic contact discontinuity in a finitely long nozzles,
the main difficulties  are caused by the facts that the states on the both sides of the contact discontinuity are unknown
and it is not clear how to locate the position of the contact discontinuity due to loss of the normal velocity on the contact discontinuity.
However, noticing that the tangent of the contact discontinuity is parallel to the velocity of the flow from the both sides,
 we can apply the Euler-Lagrange coordinates transformation to fix the free boundary and reformulate the original free boundary value problem
as a fixed boundary value problem in the Lagrangian coordinates. To solve the fixed boundary value problem,
we shall develop new ideas and arguments inspired by \cite{ly}. Since the Bernoulli  law and entropy are
invariant along the stream lines for the $C^{1}$-smooth flows, the Bernoulli's function and the
entropy can be determined completely by the incoming flows. With this observation, we  can introduce the generalized
Riemann invariants $z_{\pm}$ (see Section 3.2 below) which are invariant along the corresponding characteristics
to reduce the Euler system in the Lagrangian coordinates  to a diagonal form. 
To solve the equivalent boundary value problem for the Riemann invariants $z_{\pm}$  we shall
introduce an iteration scheme to construct a sequence of approximate solutions. Then, the remaining tasks are to
show that the iteration scheme is well-defined, and the sequence of the approximate solutions is convergent.
For this purpose we shall establish various $C^{2}$   estimates for the approximate solutions of the initial boundary value problem,
 using the characteristics method carefully for all cases depending on the reflection of the characteristics
on the nozzle walls or on the contact discontinuity. The reason we need the $C^2$-estimates is that the inhomogeneous terms in the difference of equations of
two  approximate solutions depend on the first order derivatives of the approximates solutions.
With these estimates, one can show that the iteration map is a contraction map in $C^{1}$-norm, which leads to the convergence of the
 approximate solutions. The limit is unique, and is actually the unique solution
of the nonlinear initial-boundary value problem.

The rest of the paper is organized as follows.
In section 2, we introduce the mathematical problem   and state the main results of this paper. In section 3,
we further reformulate this problem by introducing the Euler-Lagrangian coordinates transformation and the Riemann invariants
to reduce the system  to a fixed boundary value problem governed by equations of diagonalized form, and finally introduce
the iteration scheme by linearizing the nonlinear boundary value problem near the background solution. Section 4 is devoted
to the study of the linearized boundary value problem introduced in Section 3 by deriving the \emph{a prior} estimates
for the approximate solutions case by case. In section 5, we show the convergence of the approximate solutions
obtained in Section 4 and then complete the proof of the main theorem by applying the Banach contraction mapping theorem.
Finally, in the Appendix, we give an example to show that the solution for the infinity long nozzle
will blow-up in general.
\bigskip

\section{Problems and Main Result}\setcounter{equation}{0}

In this section, we shall formulate the supersonic contact discontinuity problem
in a finitely long nozzle in the Eulerian coordinates and introduce the main results.
First, we   consider a special case of a supersonic contact discontinuity
in a finitely long flat nozzle (see Fig. \ref{fig2}).
The special solution is the background solution in this paper.
 \vspace{2pt}
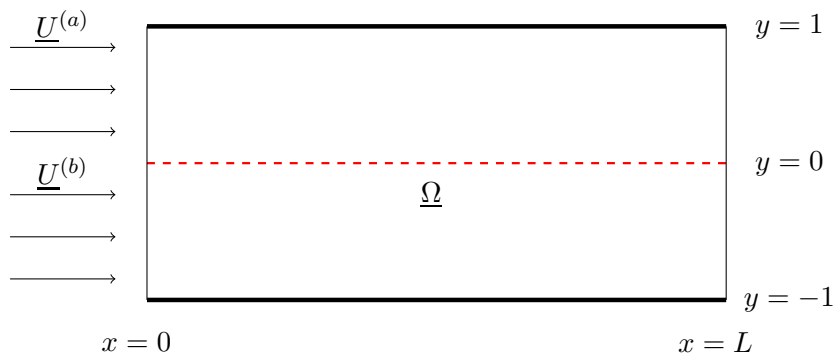
\begin{figure}[ht]
\begin{center}
\begin{tikzpicture}[scale=1.4]
\draw [thin][->] (-5,0.4) --(-4,0.4);
\draw [thin][->] (-5,0) --(-4,0);
\draw [thin][->] (-5,-0.4) --(-4,-0.4);
\draw [thin][->] (-5,-1.0) --(-4,-1.0);
\draw [thin][->] (-5.0,-1.4) --(-4,-1.4);
\draw [thin][->] (-5.0,-1.8) --(-4,-1.8);
\draw [line width=0.06cm] (-3.7,0.6) --(1.8,0.6);
\draw [line width=0.03cm][dashed][red] (-3.7,-0.7) --(1.8,-0.7);
\draw [line width=0.06cm] (-3.7,-2) --(1.8,-2);
\draw [thin] (-3.7,-2) --(-3.7,0.6);
\draw [thin] (1.8,-2) --(1.8,0.6);
\node at (2.4, -0.7) {$y=0$};
\node at (-4.5, 0.6) {$\underline{U}^{(a)}$};
\node at (-4.5, -0.8) {$\underline{U}^{(b)}$};
\node at (2.4, 0.6) {$y=1$};
\node at (2.4, -2) {$y=-1$};
\node at (-3.8, -2.4) {$x=0$};
\node at (1.7, -2.4) {$x=L$};
\node at (-1, -1) {$\underline{\Omega}$};
\end{tikzpicture}
\caption{Supersonic contact discontinuity in a finitely long straight nozzle}\label{fig2} 
\end{center}
\end{figure}

Suppose that the nozzle with flat boundaries is described as:
\begin{equation*} 
\underline{\Omega}:=\big\{(x,y)\in \mathbb{R}^{2}: 0<x<L,\ -1<y<1 \big\}.
\end{equation*}
The two layers of uniform flow in $\underline{\Omega}$ separated by a contact discontinuity  are:
 \begin{equation*}
\underline{U}^{(i)}:=(\underline{u}^{(i)}, 0,\underline{p}^{(i)}, \underline{\rho}^{(i)})^{\top},\qquad  i=a,\ b.
\end{equation*}
They are two constants states. $\underline{U}^{(a)}$ is the top layer
and $\underline{U}^{(b)}$ is the bottom layer. 
The horizontal velocity and density of both the top and bottom layers
$\underline{u}^{(i)},\ \underline{\rho}^{(i)},  i= a, b$ are positive.
The pressure $\underline{p}^{(i)}$ of both the top and bottom layers are given by the same positive constant, \emph{i.e.}, $\underline{p}^{(a)}=\underline{p}^{(b)}=\underline{p}$.
Finally, both the two layers are supersonic, \emph{i.e.}, there is a constant $\delta_{0}>0$, such that
\begin{equation*}
\underline{u}^{(i)}-\underline{c}^{(i)}>\delta_{0},
\end{equation*}
where $\underline{c}^{(i)}=\sqrt\frac{\gamma \underline{p}^{(i)}}{\underline{\rho}^{(i)}}$, for $(i= a, b)$.
 Let
\begin{eqnarray}\label{eq:2.4}
\underline{U}(x,y)=\left\{
\begin{array}{llll}
\underline{U}^{(a)},\quad (x, y)\in (0,L)\times (0,1),\\
\underline{U}^{(b)},\quad (x, y)\in (0,L)\times (-1,0).
\end{array}
     \right.
\end{eqnarray}
Then, $\underline{U}$ is a weak solution of the Euler system \eqref{eq:1.1} in $\underline{\Omega}$ in the distribution sense,
with the discontinuity line $y = 0$ as the supersonic contact discontinuity.
The line $y = 0$ divides the domain $\underline{\Omega}$ into two parts,
$\underline{\Omega}^{(i)}(i= a, b)$,
\begin{eqnarray*}
&&\underline{\Omega}^{(a)}:=\underline{\Omega}\cap\big\{0<y<1 \big\},\quad
\underline{\Omega}^{(b)}:=\underline{\Omega}\cap\big\{ -1<y<0 \big\},
\end{eqnarray*}
and the flows are supersonic in $\underline{\Omega}^{(a)}$ and in $\underline{\Omega}^{(b)}$, respectively.
We call the solution $\underline{U}(x,y)$
defined by \eqref{eq:2.4} is the background solution.

\subsection{Mathematical problems and the main results}

The incoming flow $U_{0}$
at the inlet $x=0$, which is divided by $y=0$, is given by

\begin{eqnarray}\label{eq:2.5}
U_{0}(y)=\left\{
\begin{array}{llll}
U^{(a)}_{0}(y),\quad 0<y<g_{+}(0),\\
U^{(b)}_{0}(y),\quad g_{-}(0)<y<0,
\end{array}
\right.
\end{eqnarray}
where $U^{(i)}_{0}(y)=(u^{(i)},v^{(i)}, p^{(i)}, \rho^{(i)}),\ i=a, b,$ satisfies
\begin{eqnarray}\label{eq:2.6}
\Big(\frac{v^{(a)}_{0}}{u^{(a)}_{0}}\Big)(0)=\Big(\frac{v^{(b)}_{0}}{u^{(b)}_{0}}\Big)(0), \quad \ \
p^{(a)}_{0}(0)=p^{(b)}_{0}(0).
\end{eqnarray}
So the point $(0,0)$ is the starting point of the contact discontinuity.

Let the location of the contact discontinuity be $\Gamma_{cd}=\{y=g_{cd}(x)\}$, which divides the domain $\Omega$ into two subdomains: 
\begin{eqnarray}\label{eq:2.7}
\Omega^{(a)}:=\Omega\cap\big\{g_{cd}(x)<y<g_{+}(x)\big\},\ \
\Omega^{(b)}:=\Omega\cap\big\{g_{-}(x)<y<g_{cd}(x) \big\}.
\end{eqnarray}
Let $U^{(i)}=(u^{(i)}, v^{(i)}, p^{(i)}, \rho^{(i)})$ be smooth and supersonic in $\Omega^{(i)},\ i=a, b$, respectively.

On the boundaries $\Gamma_{-},\ \Gamma_{+}$, the flow
satisfies the impermeable slip boundary conditions:
\begin{eqnarray}\label{eq:2.8}
(u, v)\cdot \mathbf{n}_{-}\big|_{\Gamma_{-}}=0,\ (u, v)\cdot \mathbf{n}_{+}\big|_{\Gamma_{+}}=0,
\end{eqnarray}
where $\mathbf{n}_{-}=(g'_{-},-1)$ and $\mathbf{n}_{+}=(-g'_{+},1)$ represent
the outer normal vectors of the low and upper boundaries
$\Gamma_{-},\ \Gamma_{+}$, respectively.
In addition, along the contact discontinuity $\Gamma_{cd}$,
the following conditions which are derived from Rankine-Hugoniot conditions hold:
\begin{eqnarray}\label{eq:2.8a}
(u, v)\cdot \mathbf{n}_{cd}\big|_{\Gamma_{cd}}=0,
\quad 
[p]\big|_{\Gamma_{cd}}=0,
\end{eqnarray}
where $\mathbf{n}_{cd}=(g'_{cd},-1)$ is the normal vector on $\Gamma_{cd}$ and $[\ ]$
denotes the jump of the quantity between the two states  across the contact discontinuity.

\par According to the above setting, we will study 
the following nonlinear free boundary problem with the supersonic contact discontinuity.

\medskip
\par $\mathbf{Problem }$ $\mathbf{A.}$ Given a supersonic incoming flow $U_{0}(y)$
at the entrance $\{x=0\}$ by \eqref{eq:2.5}, 
which satisfies \eqref{eq:2.6} and \eqref{eq:2.8}, find a piecewise smooth solutions
$(U(x,y),g_{cd}(x))$ separated by a contact discontinuity $\Gamma_{cd}$ such that, 
the flows are smooth and satisfy the Euler equations \eqref{eq:1.1} in $\Omega^{(a)}\cup\Omega^{(b)}$, the slip boundary condition \eqref{eq:2.8} on $\Gamma_{\pm}$, and the Rankine-Hugoniot condition \eqref{eq:2.8a} on  $\Gamma_{cd}$; moreover, the flows are supersonic, \emph{i.e.}, $\sqrt{u^{2}+v^{2}}>c$ in $\Omega^{(a)}\cup\Omega^{(b)}$,
where $c$ is the sonic speed. 

\medskip
We remark that a function $U(x,y)=(u,v,p,\rho)^{\top}$ of \emph{Problem A} is a weak solution
of the Euler equations \eqref{eq:1.1} in the weak sense: 
\begin{eqnarray*}
\iint_{\Omega}
W(U)\partial_x\zeta+H(U)\partial_y\zeta dxdy= 0,
\end{eqnarray*}
for any $\zeta \in C_0^{\infty}(\Omega)$, where
\begin{eqnarray*}
 W(U)=\Big(\rho u, \rho u^2+p, \rho uv, (\rho E+p)u\Big)^{\top},\quad 
H(U)=\Big(\rho v, \rho uv,\rho v^2+p, (\rho E+p)v \Big)^{\top}.
\end{eqnarray*}


\medskip
\par Our main result in this paper is the following theorem.

\begin{theorem}\label{thm:2.1}
There exist    constants $\varepsilon_{0}>0$  and   $C_{0}>0$ depending only on $\underline{U},\ L$ and $\gamma$,
such that, for any $\varepsilon \in (0, \varepsilon_{0})$, if the incoming flow and the boundaries of the nozzle satisfy
\begin{eqnarray*}
\begin{aligned}
\big\|U^{(a)}_{0}-\underline{U}^{(a)}\big\|_{C^{2}([0, g_{+}(0))]}
+\big\|U^{(b)}_{0}-\underline{U}^{(b)}\big\|_{C^{2}([0, g_{+}(0)])}\\
+\big\|g_{+}-1\big\|_{C^{3}([0, L])}
+\big\|g_{-}+1\big\|_{C^{3}([0,L])}\leq \varepsilon,
\end{aligned}
\end{eqnarray*}
then \emph{Problem A} admits a unique solution $U(x,y)$ with contact discontinuity $y=g_{cd}(x)$ satisfying:\\
\rm (i)\ Solution $U$ consists of two smooth supersonic flows $U^{(a)} \in C^{1}(\Omega^{(a)})$
and  $U^{(b)} \in C^{1}(\Omega^{(b)})$ with $y=g_{cd}(x)$ as the contact discontinuity,
and the following estimate holds:
\begin{eqnarray*}
\begin{aligned}
\big\|U^{(a)}-\underline{U}^{(a)}\big\|_{C^{1}(\Omega^{(a)})}
+\big\|U^{(b)}-\underline{U}^{(b)}\big\|_{C^{1}(\Omega^{(b)})}
\leq C_{0}\varepsilon,
\end{aligned}
\end{eqnarray*}
\rm (ii)\ The contact discontinuity $y=g_{cd}(x)$ is a stream line 
and satisfies
\begin{eqnarray*}
\big\|g_{cd}\big\|_{C^{2}([0, L])}\leq C_{0}\varepsilon,
\end{eqnarray*}
where $ g_{cd}(0)=0$.
\end{theorem}

\bigskip

\section{Mathematical Formulation and the Iteration Scheme}

\subsection{Mathematical problems in the Lagrangian coordinates }\setcounter{equation}{0}

Since the tangent of the contact discontinuity $\Gamma_{cd}$ is parallel to the velocity of   
the flow on the both sides of $\Gamma_{cd}$ by \eqref{eq:2.7}, it is convenient to apply the Euler-Lagrange
coordinates transformation to fix the free boundary $\Gamma_{cd}$ and hence reformulate \emph{Problem A}
into a fixed boundary value problem in the Lagrangian coordinates.

\par Assume that $(U(x,y), g_{cd}(x))$ is a solution of \emph{Problem A}. 
By the conservation of mass, \emph{i.e.}, $\eqref{eq:1.1}_1$,
for any $0<x<L$, it holds that
\begin{eqnarray*}
\begin{aligned}
\int^{g_{+}(x)}_{g_{cd}(x)}\rho u(x,\tau)d\tau=m^{(a)},
\quad \int^{g_{cd}(x)}_{g_{-}(x)}\rho u(x,\tau)d\tau=m^{(b)},
\end{aligned}
\end{eqnarray*}
and that
\begin{eqnarray*}
\begin{aligned}
\int^{g_{+}(x)}_{g_{-}(x)}\rho u(x,\tau)d\tau=m^{(a)}+m^{(b)},
\end{aligned}
\end{eqnarray*}
where $$m^{(a)}=\int^{g_{+}(0)}_{0}\rho^{(a)}_{0} u^{(a)}_{0}d\tau, \quad
m^{(b)}=\int^{0}_{g_{-}(0)}\rho^{(b)}_{0} u^{(b)}_{0}d\tau, $$ are the mass fluxes at the inlet above and below the contact discontinuity respectively.

\par Let
\begin{eqnarray*}
\begin{aligned}
\eta(x,y)=\int^{y}_{g_{-}(x)}\rho u(x,\tau)d\tau-m^{(b)}.
\end{aligned}
\end{eqnarray*}
By $\eqref{eq:1.1}_1$, it is easy to see that
\begin{eqnarray*}
\begin{aligned}
\frac{\partial \eta(x,y)}{\partial x}=-\rho v,\quad  \frac{\partial \eta(x,y)}{\partial y}=\rho u.
 \end{aligned}
\end{eqnarray*}
Now, we can introduce the Lagrangian coordinates transformation $\mathcal{L}$ as
\begin{eqnarray}\label{eq:3.5}
\mathcal{L} : \ \left\{
\begin{array}{llll}
\xi &= x, \\
\eta &= \eta(x,y).
\end{array}
\right.
\end{eqnarray}
Notice that
\begin{equation*}
\frac{\partial (\xi,\eta)}{\partial (x,y)}=
\Big(
\begin{array}{ccc}
  1 & 0 \\
 -\rho v & \rho u\\
\end{array}
\Big),
\end{equation*}
so the Lagrangian coordinates transformation is invertible if and only if $\rho u\neq0$, which is guaranteed in Theorem \ref{thm:3.1}, \emph{i.e.}, Remark \ref{rem:3.1}.

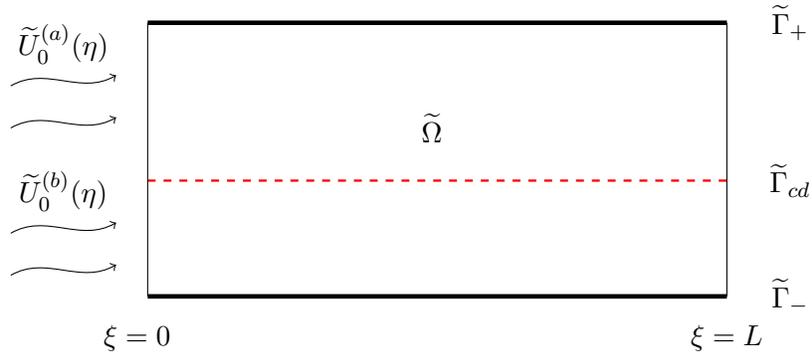
\begin{figure}[ht]
\begin{center}
\begin{tikzpicture}[scale=1.4]
\draw [thin][->] (-5,0)to[out=30, in=-150](-4,0.1);
\draw [thin][->] (-5,-0.4)to[out=30, in=-150](-4,-0.3);
\draw [thin][->] (-5.0,-1.4)to[out=30, in=-150](-4,-1.3);
\draw [thin][->] (-5.0,-1.8)to[out=30, in=-150](-4,-1.7);
\draw [line width=0.06cm] (-3.7,0.6) --(1.8,0.6);
\draw [line width=0.03cm][dashed][red] (-3.7,-0.9) --(1.8,-0.9);
\draw [line width=0.06cm] (-3.7,-2) --(1.8,-2);
\draw [thin] (-3.7,-2) --(-3.7,0.6);
\draw [thin](1.8,-2) --(1.8,0.6);
\node at (-4.5, 0.4) {$\widetilde{U}^{(a)}_{0}(\eta)$};
\node at (-4.5, -1.0) {$\widetilde{U}^{(b)}_{0}(\eta)$};
\node at (2.4, 0.6) {$\widetilde{\Gamma}_{+}$};
\node at (2.4, -0.9) {$\widetilde{\Gamma}_{cd}$};
\node at (2.4, -2) {$\widetilde{\Gamma}_{-}$};
\node at (-1, -0.4) {$\widetilde{\Omega}$};
\node at (1.8, -2.4) {$\xi=L$};
\node at (-3.8, -2.4) {$\xi=0$};
\end{tikzpicture}
\end{center}
\caption{ Supersonic contact discontinuity nozzle flows in Lagrangian coordinates}\label{fig3.1}
\end{figure}
\vspace{5pt}

Under this coordinates transformation, domain $\Omega$
becomes 
\begin{equation*}
\widetilde{\Omega}=\big\{(\xi,\eta)\in \mathbb{R}^{2}:
0<\xi<L,\ -m^{(b)}<\eta<m^{(a)} \big\},
\end{equation*}
and the lower and upper boundaries are
\begin{eqnarray*}
\begin{aligned}
\widetilde{\Gamma}_{-}:=\big\{(\xi,\eta):\eta=-m^{(b)},\ 0<\xi<L\big\},\ \
\widetilde{\Gamma}_{+}:=\big\{(\xi,\eta):\eta=m^{(b)},\  0<\xi<L\big\}.
\end{aligned}
\end{eqnarray*}
On $\Gamma_{cd}$, we have 
\begin{eqnarray*}
\eta(x, g_{cd}(x))=\int^{g_{cd}(x)}_{g_{-}(x)}\rho u(x,\tau)d\tau-m^{(b)}=0.
\end{eqnarray*}
 Thus, the free boundary $\Gamma_{cd}$ is transformed into the following fixed straight line:
\begin{eqnarray}\label{eq:3.13}
\widetilde{\Gamma}_{cd}:=\big\{(\xi,\eta):\eta=0,\ 0<\xi<L\big\}.
\end{eqnarray}
Define
\begin{eqnarray*}
\begin{aligned}
&\widetilde{\Omega}^{(a)}:=\widetilde{\Omega}\cap\{0<\eta<m^{(a)}\big\},
&\widetilde{\Omega}^{(b)}:=\widetilde{\Omega}\cap\big\{-m^{(b)}<\eta<0\big\},
\end{aligned}
\end{eqnarray*}
and let
\begin{eqnarray*}
\widetilde{U}^{(i)}(\xi,\eta)=(\widetilde{u}^{(i)},
\widetilde{v}^{(i)}, \widetilde{p}^{(i)}, \widetilde{\rho}^{(i)})^{\top}(\xi,\eta), \quad
(\xi,\eta)\in\widetilde{\Omega}^{(i)}, \  i=a, b,
\end{eqnarray*}
 be the corresponding solutions in $\widetilde{\Omega}^{(a)}$ and $\widetilde{\Omega}^{(b)}$, respectively.
%
Then the corresponding background state in the new coordinates corresponding to \eqref{eq:2.4} is
\begin{eqnarray}\label{eq:3.8}
\underline{\widetilde{U}}(\xi,\eta)=\left\{
\begin{array}{llll}
\underline{U}^{(a)},\quad (\xi, \eta)\in (0,L)\times (0, \underline{m}^{(a)}),\\
\underline{U}^{(b)},\quad (\xi, \eta)\in (0,L)\times (-\underline{m}^{(b)},0),
\end{array}
\right.
\end{eqnarray}
where $\underline{m}^{(i)}=\underline{\rho}^{(i)}\underline{u}^{(i)},\ i=a, b$.
The flow at the inlet $\xi=0$ is given by
\begin{eqnarray}\label{eq:3.9}
\widetilde{U}_{0}(\eta)=\left\{
\begin{array}{llll}
\widetilde{U}^{(a)}_{0}(\eta),\quad 0<\eta<m^{(a)},\\
\widetilde{U}^{(b)}_{0}(\eta),\quad -m^{(b)}<\eta<0,
\end{array}
\right.
\end{eqnarray}
where $\widetilde{U}^{(i)}_{0}(\eta)=(\widetilde{u}^{(i)}_{0},\widetilde{v}^{(i)}_{0},
\widetilde{p}^{(i)}_{0}, \widetilde{\rho}^{(i)}_{0})(\eta),\ i=a, b$,   satisfies
\begin{eqnarray}\label{eq:3.10}
\begin{aligned}
\Big(\frac{\widetilde{v}^{(a)}_{0}}{\widetilde{u}^{(a)}_{0}}\Big)(0)
=\Big(\frac{\widetilde{v}^{(b)}_{0}}{\widetilde{u}^{(b)}_{0}}\Big)(0), \quad
\widetilde{p}^{(a)}_{0}(0)=\widetilde{p}^{(b)}_{0}(0).
\end{aligned}
\end{eqnarray}

Notice that
$$\partial_{x} =\partial_{\xi}- \rho v \partial_{\eta},\quad
\partial_{y} =\rho u \partial_{\eta},$$
then system \eqref{eq:1.1} in the Lagrangian coordinates becomes
\begin{eqnarray}\label{eq:3.11}
\begin{aligned}
\left\{
\begin{array}{llll}
\partial_{\xi}\Big(\frac{1}{\widetilde{\rho} \widetilde{u}}\Big)
- \partial_{\eta}\Big(\frac{\widetilde{v}}{\widetilde{u}}\Big)=0,\\
\partial_{\xi}\Big(\widetilde{u}+\frac{\widetilde{p}}{\widetilde{\rho} \widetilde{u}}\Big)
-\partial_{\eta}\Big(\frac{\widetilde{p} \widetilde{v}}{\widetilde{u}}\Big)=0,\\
\partial_{\xi}\widetilde{v}+\partial_{\eta}\widetilde{p}=0,\\
\widetilde{p}=A(\widetilde{S})\widetilde{\rho}^{\gamma-1},
\end{array}
\right.
\end{aligned}
\end{eqnarray}
together with the Bernoulli law:
\begin{eqnarray}\label{eq:3.12}
\frac{1}{2}(\widetilde{u}^2+\widetilde{v}^2)+\frac{\gamma \widetilde{p}}{(\gamma-1)\widetilde{\rho}}
=\left\{
\begin{array}{llll}
\widetilde{B}^{(a)}_{0}(\eta) ,    &\ \ \ \ (\xi,\eta)\in \widetilde{\Omega}^{(a)}, \\
\widetilde{B}^{(b)}_{0}(\eta),  &\ \ \ \ (\xi,\eta)\in \widetilde{\Omega}^{(b)}.
\end{array}
\right.
\end{eqnarray}
Here we use the fact that $\widetilde{B}^{(i)}_{0}(\eta)$ 
for $i= a, b$ are conserved along the streamlines by $\eqref{eq:1.1}_4$, thus they depend only on  $\gamma$
and the incoming flow $\widetilde{U}_{0}(\eta)$ at the entrance $\xi=0$.

 The boundary conditions in \eqref{eq:2.7} become
\begin{eqnarray}\label{eq:3.16}
\frac{\widetilde{v}^{(a)}}{\widetilde{u}^{(a)}}\Big|_{\widetilde{\Gamma}_{+}}=g'_{+}(\xi),
\quad\quad \frac{\widetilde{v}^{(b)}}{\widetilde{u}^{(b)}}\Big|_{\widetilde{\Gamma}_{-}}=g'_{-}(\xi),
\end{eqnarray}
and the Rankine-Hugoniot conditions on  $\widetilde{\Gamma}_{cd}$ read as
\begin{eqnarray}\label{eq:3.17}
\frac{\widetilde{v}^{(a)}}{\widetilde{u}^{(a)}}\Big|_{\widetilde{\Gamma}_{cd}}
=\frac{\widetilde{v}^{(b)}}{\widetilde{u}^{(b)}}\Big|_{\widetilde{\Gamma}_{cd}}=g'_{cd}(\xi), \qquad
\widetilde{p}^{(a)}\big|_{\widetilde{\Gamma}_{cd}}=\widetilde{p}^{(b)}\big|_{\widetilde{\Gamma}_{cd}},
\end{eqnarray}
 which indicates that $\frac{\widetilde{v}}{\widetilde{u}}$ and $\widetilde{p}$ are continuous
  across  $\widetilde{\Gamma}_{cd}$.

\par Then, the free boundary value problem, \emph{Problem A}, in the Eulerian coordinates 
can be reformulated as the following nonlinear fixed boundary value problem in the Lagrangian coordinates.

\par $\mathbf{Problem }$ $\mathbf{B.}$ Given a supersonic incoming flow $\widetilde{U}_{0}(\eta)$
at the entrance $\{\xi=0\}$ by \eqref{eq:3.9} 
satisfying \eqref{eq:3.10} and \eqref{eq:3.16}, find a piecewise smooth solutions
$\widetilde{U}(\xi,\eta)$ with a contact discontinuity along the straight line $\widetilde{\Gamma}_{cd}$ such that,
the flows are smooth and supersonic,  and satisfy the Euler equation \eqref{eq:3.11} and \eqref{eq:3.12} in $\widetilde{\Omega}^{(a)}\cup\widetilde{\Omega}^{(b)}$, the slip boundary condition \eqref{eq:3.16} on $\widetilde{\Gamma}_{\pm}$, and the Rankine-Hugoniot condition \eqref{eq:3.17} on  $\widetilde{\Gamma}_{cd}$. 

Thus Theorem \ref{thm:2.1} becomes the following theorem: 
\begin{theorem}\label{thm:3.1}
There exist 
two constants $\tilde{\varepsilon}_{0}>0$  and   $\widetilde{C}_{0}>0$
depending only on $\underline{\widetilde{U}}, \ L$ and  $\gamma$, such that for any
$\tilde{\varepsilon} \in (0, \tilde{\varepsilon}_{0})$
if  
\begin{eqnarray}\label{eq:3.19}
\begin{aligned}
\big\|\widetilde{U}^{(a)}_{0}-\underline{U}^{(a)}\big\|_{C^{2}([0, m^{(a)}])}
+\big\|\widetilde{U}^{(b)}_{0}-\underline{U}^{(b)}\big\|_{C^{2}([-m^{(b)},0])}\\
 +\big\|g_{-}+1\big\|_{C^{3}([0,L])}
+\big\|g_{+}-1\big\|_{C^{3}([0,L])}\leq \tilde{\varepsilon},
\end{aligned}
\end{eqnarray}
then \emph{Problem B} 
admits a unique piecewise smooth solution
$\widetilde{U}(\xi,\eta)$ 
consisting of
two smooth supersonic flow $\widetilde{U}^{(a)} \in C^{1}(\widetilde{\Omega}^{(a)})$
and $\widetilde{U}^{(b)} \in C^{1}(\widetilde{\Omega}^{(b)})$ with $\eta=0$ as the discontinuity.
Moreover, 
\begin{eqnarray}\label{eq:3.22}
\begin{aligned}
\big\|\widetilde{U}^{(a)}-\underline{U}^{(a)}\big\|_{C^{1}(\widetilde{\Omega}^{(a)})}
+\big\|\widetilde{U}^{(b)}-\underline{U}^{(b)}\big\|_{C^{1}(\widetilde{\Omega}^{(b)})}\leq \widetilde{C}_{0}\tilde{\varepsilon}.
\end{aligned}
\end{eqnarray}

\end{theorem}

\begin{remark}\label{rem:3.1}
Theorem \ref{thm:3.1} and Theorem \ref{thm:2.1} are equivalent  under the Lagrangian coordinates transformation $\mathcal{L}$.
The reason is that for small $\tilde{\varepsilon}$,   we have
\begin{equation*}
\det\Big(\frac{\partial (\xi,\eta)}{\partial (x,y)}\Big)=\rho u>0.
\end{equation*}
So the inverse Lagrangian coordinates transformation $\mathcal{L}^{-1}$ exists and can be given explicitly by
\begin{eqnarray*}
\mathcal{L}^{-1} : \ \left\{
\begin{array}{llll}
  x &= \xi, \\
   y &= \int^{\eta}_{-m^{(b)}}\Big(\frac{1}{\widetilde{\rho} \widetilde{u}}\Big)(\xi,\tau)d\tau+g_{-}(\xi).
     \end{array}
     \right.
\end{eqnarray*}
Therefore, if Theorem \ref{thm:3.1} holds, then one can  use  $\mathcal{L}^{-1}$ 
to define $U(x,y)=\widetilde{U}(\xi(x,y), \eta(x,y))$ and define 
\begin{eqnarray*}
g_{cd}(x)=\int^{0}_{-m^{(b)}}\Big(\frac{1}{\rho u}\Big)(x,\tau)d\tau+g_{-}(x), \  x\in [0, L].
\end{eqnarray*}
Obviously,
\begin{eqnarray*}
g'_{cd}(x)=\frac{v}{u}(x,g_{cd}(x)), \  x\in [0,L],
\end{eqnarray*}
and $g_{cd}'(x)\in C^{2}([0,L])$. It concludes the proof of Theorem \ref{thm:2.1}.
Therefore, in the rest of the paper, we only need to consider \emph{Problem B} and prove  
Theorem \ref{thm:3.1}.
\end{remark}

\subsection{Riemann invariants}
In this subsection, we will use the Riemann invariants to diagonalize the system \eqref{eq:3.11}.

\par First, the system \eqref{eq:3.11}  can be rewritten as the following 
first order non-divergence symmetric system.
\begin{eqnarray}\label{eq:3.24}
A(\widetilde{U})\partial_{\xi}\widetilde{U}+B(\widetilde{U})\partial_{\eta}\widetilde{U}=0,
\end{eqnarray}
where $\widetilde{U}(\xi, \eta)=(\widetilde{u}, \widetilde{v}, \widetilde{p})^{\top}(\xi, \eta)$
and
\begin{equation*}
A(\widetilde{U})=\left(
\begin{array}{ccc}
\widetilde{u} &0 & \frac{1}{\widetilde{\rho}} \\
0& \widetilde{u} & 0 \\
\frac{1}{\widetilde{\rho}} & 0 & \frac{\widetilde{u}}{\widetilde{c}^{2}\widetilde{\rho}^{2}}
\end{array}
\right), \quad
 B(\widetilde{U})=\left(
\begin{array}{ccc}
0 &0 & -\widetilde{v} \\
0& 0 & \widetilde{u} \\
-\widetilde{v} & \widetilde{u} & 0
\end{array}
\right).
\end{equation*}
%
The eigenvalues of \eqref{eq:3.24} are
\begin{eqnarray*}
\begin{aligned}
\lambda_{-}=\frac{\widetilde{\rho} \widetilde{u} \widetilde{c}^{2}}
{\widetilde{u}^{2}-\widetilde{c}^{2}}\Big(\frac{\widetilde{v}}{\widetilde{u}}
-\frac{\sqrt{\widetilde{u}^{2}+\widetilde{v}^{2}-\widetilde{c}^{2}}}{\widetilde{c}}\Big),\quad
\lambda_{0}=0,  \quad
\lambda_{+}=\frac{\widetilde{\rho} \widetilde{u} \widetilde{c}^{2}}
{\widetilde{u}^{2}-\widetilde{c}^{2}}\Big(\frac{\widetilde{v}}{\widetilde{u}}
+\frac{\sqrt{\widetilde{u}^{2}+\widetilde{v}^{2}-\widetilde{c}^{2}}}{\widetilde{c}}\Big),
\end{aligned}
\end{eqnarray*}
 and the associated right and left eigenvectors are
\begin{eqnarray*}
\begin{aligned}
r_{-}=\Big(\frac{\lambda_{-}}{\widetilde{\rho}}+\widetilde{v}, -\widetilde{v}, -\lambda_{-}\widetilde{u}\Big)^{\top},\quad
r_{0}=\big(\widetilde{u},\widetilde{v}, 0\big)^{\top},   \quad
r_{+}=\Big(\frac{\lambda_{+}}{\widetilde{\rho}}+\widetilde{v}, -\widetilde{v},
 -\lambda_{+}\widetilde{u}\Big)^{\top},
\end{aligned}
\end{eqnarray*}
and
\begin{eqnarray*}
l_{\pm}=(r_{\pm})^{\top}, \quad l_{0}=(r_{0})^{\top}.
\end{eqnarray*}
Multiply system \eqref{eq:3.24} by $l_{\pm}$ and $l_{0}$ to get
\begin{eqnarray}\label{eq:3.29}
\begin{aligned}
\widetilde{v}(\partial_{\xi}\widetilde{u}+\lambda_{\pm}\partial_{\eta}\widetilde{u})
-\widetilde{u}(\partial_{\xi}\widetilde{v}+\lambda_{\pm}\partial_{\eta}\widetilde{v})
\mp\frac{\sqrt{\widetilde{u}^{2}+\widetilde{v}^{2}-\widetilde{c}^{2}}}{\widetilde{\rho} \widetilde{c}}(\partial_{\xi}\widetilde{p}+\lambda_{\pm}\partial_{\eta}\widetilde{p})=0,
\end{aligned}
\end{eqnarray}
and
\begin{eqnarray}\label{eq:3.30}
\widetilde{u}\partial_{\xi}\widetilde{u}+\widetilde{v}\partial_{\xi}\widetilde{v}
+\frac{\partial_{\xi}\widetilde{p}}{\widetilde{\rho}}=0.
\end{eqnarray}
By the Bernoulli  law \eqref{eq:3.12}, \eqref{eq:3.30} can be reduced to
\begin{eqnarray*}
\partial_{\xi}\widetilde{S}=0,
\end{eqnarray*}
which implies that $\widetilde{S}$ is constant in $\widetilde{\Omega}^{(a)}$ or $\widetilde{\Omega}^{(b)}$, \emph{i.e.},
\begin{eqnarray}\label{eq:3.32}
\widetilde{S}=\widetilde{S}_{0}(\eta).
\end{eqnarray}
\par Denote
\begin{eqnarray*}
\widetilde{w}:=\frac{\widetilde{v}}{\widetilde{u}},
\ \ \ \Lambda:=\frac{\sqrt{\widetilde{u}^{2}+\widetilde{v}^{2}
-\widetilde{c}^{2}}}{\widetilde{\rho} \widetilde{c}\widetilde{u}^{2}},
\end{eqnarray*}
and define the operator,
\begin{eqnarray*}
\mathscr{D}_{-}=\partial_{\xi}+\lambda_{-}\partial_{\eta}, \quad
\mathscr{D}_{+}=\partial_{\xi}+\lambda_{+}\partial_{\eta}.
\end{eqnarray*}
Then, we can further rewrite the equations
\eqref{eq:3.29} as 
\begin{eqnarray}\label{eq:3.35}
\mathscr{D}_{-}\widetilde{w}-\Lambda \mathscr{D}_{-}\widetilde{p}=0, \quad
\mathscr{D}_{+}\widetilde{w}+\Lambda \mathscr{D}_{+}\widetilde{p}=0.
\end{eqnarray}

\begin{remark}\label{rem:3.2}
Once $\widetilde{w}$ and $\widetilde{p}$ are solved, then by the Bernoulli law \eqref{eq:3.12}, we can obtain
 $\widetilde{u}$ and $\widetilde{v}$ as the following:
\begin{eqnarray*}
\begin{aligned}
&\widetilde{u}=\sqrt{\frac{2\Big((\gamma-1)\widetilde{B}_{0}-\gamma \widetilde{A}_{0}^{\frac{1}{\gamma}}\widetilde{p}^{\frac{\gamma-1}{\gamma}}\Big)}
{(\gamma-1)\big(1+\widetilde{w}^{2}\big)}},\quad
\widetilde{v}=\widetilde{w}\sqrt{\frac{2\Big((\gamma-1)\widetilde{B}_{0}
-\gamma\widetilde{A}_{0}^{\frac{1}{\gamma}}
\widetilde{p}^{\frac{\gamma-1}{\gamma}}\Big)}{(\gamma-1)\big(1+\widetilde{w}^{2}\big)}},
\end{aligned}
\end{eqnarray*}
where $\widetilde{A}_{0}(\eta)=A(\widetilde{S}_{0}(\eta))$
and $\widetilde{S}_{0}(\eta)$ is given by \eqref{eq:3.32}.
\end{remark}

 \par Therefore, we only need to solve $\widetilde{w}$ and $\widetilde{p}$. Set $\mathscr{W}=(\widetilde{w},\widetilde{p})^{\top}$,
then the system \eqref{eq:3.35} can be rewritten in the following:
 \begin{eqnarray}\label{eq:3.37}
\partial_{\xi}\mathscr{W}+\mathscr{A}\partial_{\eta}\mathscr{W}=0,
\end{eqnarray}
 where
\begin{eqnarray}\label{eq:3.38}
\begin{aligned}
\mathscr{A}=\left(
\begin{array}{ccc}
\frac{\widetilde{\rho} \widetilde{c}^{2}\widetilde{v}}{\widetilde{u}^{2}-\widetilde{c}^{2}} & \frac{\widetilde{u}^{2}+\widetilde{v}^{2}-\widetilde{c}^{2}}{\widetilde{u}
(\widetilde{u}^{2}-\widetilde{c}^{2})} \\[8pt]
\frac{\widetilde{\rho}^{2} \widetilde{c}^{2}\widetilde{u}^{3}}
{\widetilde{u}^{2}-\widetilde{c}^{2}}  & \frac{\widetilde{\rho} \widetilde{c}^{2}\widetilde{v}}
{\widetilde{u}^{2}-\widetilde{c}^{2}}
\end{array}
\right).
\end{aligned}
\end{eqnarray}
Direct computation shows that the eigenvalues of \eqref{eq:3.38} are $\lambda_{\pm}$
 and the corresponding right eigenvectors $\tilde{r}_{\pm}$ are
\begin{eqnarray*}
&& \tilde{r}_{\pm}= \big(\sqrt{\widetilde{u}^{2}+\widetilde{v}^{2}-\widetilde{c}^{2}},
\pm\widetilde{\rho} \widetilde{c}\widetilde{u}^{2}\big)^{\top}.
\end{eqnarray*}
Then we can define the Riemann invariants $z_{\pm}$ for the system \eqref{eq:3.37} as the following form:
\begin{eqnarray}\label{eq:3.40a}
z_{-}=\arctan\widetilde{w}+\Theta(\widetilde{p}; \widetilde{S}_{0}, \widetilde{B}_{0}),
\ \ z_{+}=\arctan\widetilde{w}-\Theta(\widetilde{p}; \widetilde{S}_{0}, \widetilde{B}_{0}),
\end{eqnarray}
where
\begin{eqnarray}\label{eq:3.41}
\Theta(\widetilde{p}; \widetilde{S}_{0}, \widetilde{B}_{0})
=\int^{\widetilde{p}}\frac{\sqrt{2 \widetilde{B}_{0}-\frac{\gamma(\gamma+1)}
{\gamma-1}\widetilde{A}_{0}^{\frac{1}{\gamma}}
\tau^{\frac{\gamma-1}{\gamma}}}}{2\gamma^{\frac{1}{2}}\widetilde{A}_{0}^{-\frac{1}{2\gamma}}
\Big(\widetilde{B}_{0}-\frac{\gamma}{\gamma-1} \widetilde{A}_{0}^{\frac{1}{\gamma}}\tau^{1-\frac{1}{\gamma}}\Big)\tau^{\frac{\gamma+1}{2\gamma}}}d\tau.
\end{eqnarray}
By \eqref{eq:3.40a} and \eqref{eq:3.41}, we have
\begin{eqnarray}\label{eq:3.40}
\begin{aligned}
\widetilde{w}=\tan(\frac{z_{-}+z_{+}}{2}), \quad \Theta(\widetilde{p}; \widetilde{S}_{0}, \widetilde{B}_{0})
=\frac{1}{2}(z_{-}-z_{+}).
\end{aligned}
\end{eqnarray}

Set $z=(z_{-},z_{+})^{\top} $. By   the implicit function theorem, we have the following lemma.
\begin{lemma}\label{lem:3.1}
For any given $z$, if the flow is supersonic,  
then equation \eqref{eq:3.40}
admits a unique solution $\widetilde{p}=\widetilde{p}\big(z; \widetilde{S}_{0}, \widetilde{B}_{0}\big)$ and $\widetilde{w}=\widetilde{w}(z)$.
\end{lemma}
\begin{proof}
By the straightforward computation, we have
\begin{eqnarray}\label{eq:3.48a}
\begin{aligned}
\frac{\partial \widetilde{p}}{\partial z_{-}}=\frac{1}{2\partial_{\widetilde{p}}
\Theta(\widetilde{p}; \widetilde{S}_{0}, \widetilde{B}_{0})}, \quad
\frac{\partial \widetilde{p}}{\partial z_{+}}=-\frac{1}{2\partial_{\widetilde{p}}
\Theta(\widetilde{p}; \widetilde{S}_{0}, \widetilde{B}_{0})},
\end{aligned}
\end{eqnarray}
where
\begin{eqnarray}\label{eq:3.48b}
\partial_{\widetilde{p}}\Theta(\widetilde{p}; \widetilde{S}_{0}, \widetilde{B}_{0})
=\frac{\sqrt{2 \widetilde{B}_{0}-\frac{\gamma(\gamma+1)}{\gamma-1}\widetilde{A}_{0}^{\frac{1}{\gamma}}
\widetilde{p}^{\frac{\gamma-1}{\gamma}}}}{2\gamma^{\frac{1}{2}}\widetilde{A}_{0}^{-\frac{1}{2\gamma}}
\Big(\widetilde{B}_{0}-\frac{\gamma}{\gamma-1}\widetilde{A}_{0}^{\frac{1}{\gamma}}
\widetilde{p}^{1-\frac{1}{\gamma}}\Big)\widetilde{p}^{\frac{\gamma+1}{2\gamma}}}>0.
\end{eqnarray}
Then the lemma follows from the implicit function theorem.
\end{proof}

\par By Remark \ref{rem:3.2} and Lemma \ref{lem:3.1},  we only need to consider the following nonlinear boundary
value problem: 
\begin{eqnarray}\label{eq:3.51b}
(\widetilde{\mathbf{P}})\quad   \left\{
\begin{array}{llll}
\partial_{\xi}z^{a}+diag( \lambda^{a}_{+}, \lambda^{a}_{-})\partial_{\eta}z^{a}=0,
&\ \ \ in\ \widetilde{\Omega}^{(a)},   \\
\partial_{\xi}z^{b}+diag( \lambda^{b}_{+}, \lambda^{b}_{-})\partial_{\eta}z^{b}=0,
 &\ \ \ in\ \widetilde{\Omega}^{(b)},  \\
z^{a}=z^{a}_{0}(\eta),  &\ \ \  on\ \xi=0,\\
z^{b}=z^{b}_{0}(\eta),   &\ \ \  on\ \xi=0, \\
z^{a}_{-}+z^{a}_{+}=2\arctan g'_{+}(\xi), &\ \ \ on\ \widetilde{\Gamma}_{+},\\
z^{b}_{-}+z^{b}_{+}=2\arctan g'_{-}(\xi), &\ \ \ on\ \widetilde{\Gamma}_{-},\\
z^{a}_{-}+z^{a}_{+}=z^{b}_{-}+z^{b}_{+}, &\ \ \ on\ \widetilde{\Gamma}_{cd},\\
\widetilde{p}\big(z^{a}; \widetilde{S}^{a}_{0}, \widetilde{B}^{a}_{0}\big)
=\widetilde{p}\big(z^{b}; \widetilde{S}^{b}_{0}, \widetilde{B}^{b}_{0}\big),
&\ \ \ on\ \widetilde{\Gamma}_{cd}.
\end{array}
\right.
\end{eqnarray}
Here, $z^{(i)},\lambda^{(i)}_{\pm}, \widetilde{S}^{(i)}_{0}, \widetilde{B}^{(i)}_{0}, i=a, b, $
represent the states taking values in $\widetilde{\Omega}^{(a)}$ and $\widetilde{\Omega}^{(b)}$,
respectively.

Then \emph{Problem B} is reformulated into the following problem.

\medskip
\par $\mathbf{Problem }$ $\mathbf{C.}$ Given a supersonic
incoming flow $\big(z^{a}_{0}, z^{b}_{0}\big)(\eta)$ at the entrance $\{\xi=0\}$
satisfying \eqref{eq:3.10} and \eqref{eq:3.17}, find a piecewise smooth supersonic flow $\big(z^{a}, z^{b}\big)(\xi,\eta)$ of the nonlinear boundary value problem
$(\widetilde{\mathbf{P}})$. 

\medskip
\par Then, we have the following theorem for \emph{Problem  C}.
\begin{theorem}\label{thm:3.2}
There exist 
two constants $\tilde{\varepsilon}'_{0}>0$   and   $\widetilde{C}_{1}>0$ depending
only on $\underline{\widetilde{U}}, L$ and $\gamma$ such that, for each $\tilde{\varepsilon}_{1}
\in (0, \tilde{\varepsilon}'_{0})$, if
\begin{eqnarray}\label{eq:3.53}
\begin{aligned}
&\sum_{i=a,b}\Big(\|z^{i}_{0}-\underline{z}^{i}\|_{C^{2}(\Sigma_{i})}
+\|\widetilde{S}^{i}_{0}-\underline{S}^{i}\|_{C^{2}( \Sigma_{i})}
+\|\widetilde{B}^{i}_{0}-\underline{B}^{i}\|_{C^{2}( \Sigma_{i})}\Big)\\
&\ \quad\quad\quad\quad\quad\quad\quad\quad\quad\quad\quad\quad
+\|g_{+}-1\|_{C^{3}([0, L])}+\|g_{-}+1\|_{C^{3}([0, L])}\leq\tilde{\varepsilon}_{1},
\end{aligned}
\end{eqnarray}
where $\Sigma_{a}=(0,m^{(a)})$, $\Sigma_{b}=(-m^{(b)},0)$,
then \emph{Problem C} admits a unique solution
$(z^{a},z^{b} )\in C^{1}(\widetilde{\Omega}^{(a)})\times C^{1}(\widetilde{\Omega}^{(b)})$
satisfying
\begin{eqnarray}\label{eq:3.56}
\begin{aligned}
\big\|z^{a}-\underline{z}^{a}\big\|_{C^{1}(\widetilde{\Omega}^{(a)})}
+\big\|z^{b}-\underline{z}^{b}\big\|_{C^{1}(\widetilde{\Omega}^{(b)})}
\leq \widetilde{C}_{1}\tilde{\varepsilon}_{1}.
\end{aligned}
\end{eqnarray}
\end{theorem}

\begin{remark}\label{rem:3.3}
By Remark \ref{rem:3.2} and Lemma \ref{lem:3.1}, Theorem \ref{thm:3.1} directly follows from Theorem \ref{thm:3.2}.
Thus, to solve \emph{Problem B}, it suffices to
show the existence and uniqueness of solutions of problem $(\widetilde{\mathbf{P}})$.  The next several
sections are devoted to the study of the non-linear boundary value problem $(\widetilde{\mathbf{P}})$.
\end{remark}

\subsection{Iteration scheme}

\vspace{5pt}
\begin{figure}[ht]
\begin{center}\label{fig3.2}
\begin{tikzpicture}[scale=1.5]
\draw [thin][->] (-5,0)to[out=30, in=-150](-4,0.1);
\draw [thin][->] (-5,-0.4)to[out=30, in=-150](-4,-0.3);
\draw [thin][->] (-5.0,-1.4)to[out=30, in=-150](-4,-1.3);
\draw [thin][->] (-5.0,-1.8)to[out=30, in=-150](-4,-1.7);
\draw [line width=0.06cm] (-3.7,0.6) --(2.5,0.6);
\draw [line width=0.06cm][red] (-3.7,-0.5) --(-0.7,-0.5);
\draw [line width=0.06cm][dashed][red] (-0.7,-0.5) --(2.5,-0.5);
\draw [line width=0.06cm] (-3.7,-2) --(2.5,-2);
\draw [thin] (-3.7,-2) --(-3.7,0.6);
\draw [line width=0.06cm](2.5,-2) --(2.5,0.6);
\draw [thin][blue](-3.7,-0.5) --(-2.6,0.6);
\draw [thin][blue](-3.7,-0.1) --(-3.0,0.6);
\draw [thin][blue](-3.7,0.3) --(-3.4,0.6);
\draw [thin][red](-3.7,0.6) --(-2.6,-0.5);
\draw [thin][red](-3.7,0.3) --(-2.9,-0.5);
\draw [thin][red](-3.7,-0.1) --(-3.3,-0.5);
\draw [thin][black](-2.6, 0.6) --(-1.5,-0.5);
\draw [thin][black](-3.0,0.6) --(-1.9,-0.5);
\draw [thin][black](-3.4,0.6) --(-2.3,-0.5);
\draw [thin][orange](-1.5,-0.5) --(-0.4,0.6);
\draw [thin][orange](-1.9,-0.5) --(-0.8,0.6);
\draw [thin][orange](-2.3,-0.5) --(-1.2,0.6);
\draw [thin][orange](-2.6,-0.5) --(-1.5,0.6);
\draw [thin][orange](-2.9,-0.5) --(-1.8,0.6);
\draw [thin][orange](-3.3,-0.5) --(-2.2,0.6);
\draw [thin][green](-2.2,0.6) --(-1.1,-0.5);
\draw [thin][green](-1.8,0.6) --(-0.7,-0.5);
\draw [thin][green](-1.5,0.6) --(-0.4,-0.5);
\draw [thin][green](-1.2,0.6)--(-0.1,-0.5);
\draw [thin][green](-0.8,0.6)--(0.3,-0.5);
\draw [thin][green](-0.4,0.6) --(0.7,-0.5);
\draw [thin][orange](-1.1,-0.5) --(0,0.6);
\draw [thin][orange](-0.7,-0.5)--(0.4,0.6);
\draw [thin][orange](-0.4,-0.5)--(0.7,0.6);
\draw [thin][orange](-0.1,-0.5)--(1.0,0.6);
\draw [thin][orange](0.3,-0.5)--(1.4,0.6);
\draw [thin][orange](0.7,-0.5)--(1.8,0.6);
\draw [thin][black](0,0.6) --(1.1,-0.5);
\draw [thin][black](0.4,0.6)--(1.5,-0.5);
\draw [thin][black](0.7,0.6)--(1.8,-0.5);
\draw [thin][black](1.0,0.6)--(2.0,-0.4);
\draw [thin][black](1.4,0.6)--(2.0,0);
\draw [thin][black](1.8,0.6)--(2.0,0.4);
\draw [thin][red](1.1,-0.5)--(2.0,0.4);
\draw [thin][red](1.5,-0.5)--(2.0,0);
\draw [thin][red](1.8,-0.5)--(2.0,-0.3);
\draw [thin][green](-3.7,-2) --(-2.3,-0.5);
\draw [thin][green](-3.7,-1.5) --(-2.7,-0.5);
\draw [thin][green] (-3.7,-1) --(-3.2,-0.5);
\draw [thin][blue](-3.7,-0.5) --(-2.2,-2);
\draw [thin][blue](-3.7,-1) --(-2.7,-2);
\draw [thin][blue] (-3.7,-1.5) --(-3.2,-2);
\draw [thin][orange] (-3.2,-2) --(-1.7,-0.5);
\draw [thin][orange] (-2.7,-2) --(-1.2,-0.5);
\draw [thin][orange] (-2.2,-2) --(-0.7,-0.5);
\draw [thin][black] (-3.2,-0.5) --(-1.7,-2);
\draw [thin][black] (-2.7,-0.5) --(-1.2,-2);
\draw [thin][black] (-2.2,-0.5) --(-0.7,-2);
\draw [thin][black] (-1.7,-0.5) --(-0.2,-2);
\draw [thin][black] (-1.2,-0.5) --(0.3,-2);
\draw [thin][black] (-0.7,-0.5) --(0.8,-2);
\draw [thin][red] (-1.7,-2) --(-0.7,-1);
\draw [thin][red] (-1.2,-2) --(-0.7,-1.5);
\draw [line width=0.03cm][dashed][red] (-0.7,-2) --(-0.7,0.6);
\draw [thin][green](-0.7,-2) --(0.8,-0.5);
\draw [thin][green] (-0.7,-1.5) --(0.3,-0.5);
\draw [thin][green](-0.7,-1) --(-0.2,-0.5);
\draw [thin][red](0.8,-0.5) --(2.0,-1.7);
\draw [thin][red] (0.3,-0.5) --(1.8,-2);
\draw [thin][red](-0.2,-0.5) --(1.3,-2);
\draw [thin][blue] (-0.2,-2)--(1.3,-0.5);
\draw [thin][blue] (0.3,-2) --(1.8,-0.5);
\draw [thin][blue] (0.8,-2) --(2.0,-0.8);
\draw [thin][blue] (1.3,-2) --(2.0,-1.3);
\draw [thin][blue] (1.8,-2) --(2.0,-1.8);
\draw [thin][red] (1.3,-0.5) --(2.0,-1.2);
\draw [thin][red](1.8,-0.5) --(2.0,-0.7);
\draw [line width=0.03cm][dashed][red] (2.0,-2) --(2.0,0.6);
\node at (-4.5, 0.4) {$\widetilde{U}^{(a)}_{0}(\eta)$};
\node at (-4.5, -1.0) {$\widetilde{U}^{(b)}_{0}(\eta)$};
\node at (2.8, 0.6) {$\widetilde{\Gamma}_{+}$};
\node at (2.8, -0.5) {$\widetilde{\Gamma}_{cd}$};
\node at (2.8, -2) {$\widetilde{\Gamma}_{-}$};
\node at (-2.0, 0.1) {$\widetilde{\Omega}^{(a)}$};
\node at (-2.0, -1) {$\widetilde{\Omega}^{(b)}$};
\node at (2.8, -2.3) {$\xi=L$};
\node at (-0.7, -2.3) {$\xi=\xi^{*}_{1}$};
\node at (1.8, -2.3) {$\xi=\xi^{*}_{2}$};
\node at (-3.8, -2.3) {$\xi=0$};
\end{tikzpicture}
\end{center}
\caption{Iteration scheme for the boundary value problem $(\widetilde{\mathbf{P}})$}\label{fig3.2}
\end{figure}
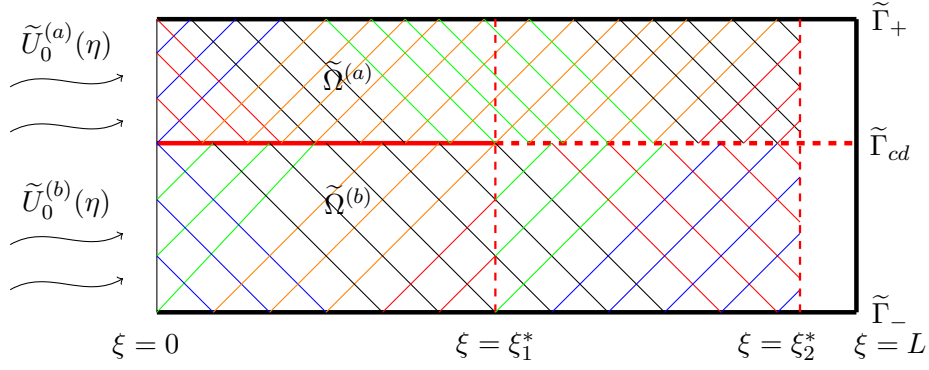
\vspace{5pt}

\par In order to solve the nonlinear boundary value problem: \emph{Problem C},  we first linearize the problem $(\widetilde{\mathbf{P}})$ to construct a sequence of approximate solutions, and then show that the approximate solutions are convergent and their limit is actually the solution of \emph{Problem C}.
The linearized hyperbolic boundary value problem for $z$ is solved in $\widetilde{\Omega}^{(a)}$ and $\widetilde{\Omega}^{(b)}$
at the same time (see Fig \ref{fig3.2}). 

First, the linearized boundary value problem of $(\widetilde{\mathbf{P}})$ is
\begin{eqnarray}\label{eq:4.1a}
(\mathbf{\widetilde{P}_{n}})\quad
\left\{
\begin{array}{llll}
\partial_{\xi}\delta z^{a,(n)}+diag( \tilde{\lambda}^{a,(n-1)}_{+}, \tilde{\lambda}^{a,(n-1)}_{-})
\partial_{\eta}\delta z^{a,(n)}=0, &\ \ \ in\ \widetilde{\Omega}^{(a)},  \\
\partial_{\xi}\delta z^{b,(n)}+diag( \tilde{\lambda}^{b,(n-1)}_{+}, \tilde{\lambda}^{b,(n-1)}_{-})
\partial_{\eta}\delta z^{b,(n)}=0, &\ \ \ in\ \widetilde{\Omega}^{(b)},  \\
\delta z^{a,(n)}=\delta z^{a}_{0}(\eta),  &\ \ \  on\ \xi=0,\\
\delta z^{b, (n)}=\delta z^{b)}_{0}(\eta),   &\ \ \  on\ \xi=0, \\
\delta z^{a,(n)}_{-}+\delta z^{a,(n)}_{+}=2\arctan g'_{+}(\xi), &\ \ \ on\ \widetilde{\Gamma}_{+},\\
\delta z^{b,(n)}_{-}+\delta z^{b,(n)}_{+}=2\arctan g'_{-}(\xi), &\ \ \ on\ \widetilde{\Gamma}_{-},\\
\delta z^{a, (n)}_{-}-\delta z^{b, (n)}_{+}=\delta z^{b, (n)}_{-}-\delta z^{a, (n)}_{+},
&\ \ \ on\ \widetilde{\Gamma}_{cd},\\
\alpha^{(n-1)}\delta z^{a,(n)}_{-}+\beta^{(n-1)}\delta z^{b,(n)}_{+}\\
\ \ \ =\alpha^{(n-1)}\delta z^{a,(n)}_{+}+\beta^{(n-1)}\delta z^{b,(n)}_{-}+c(\xi),
&\ \ \ on\ \widetilde{\Gamma}_{cd},
\end{array}
\right.
\end{eqnarray}
where
\begin{eqnarray}\label{eq:3.47}
\begin{aligned}
&\delta z^{i,(n)}:=z^{(n)}-\underline{z}^{i},
\ \ \ \delta z^{i}_{0}(\eta):=z^{i}_{0}(\eta)-\underline{z}^{i},\  \ \
\tilde{\lambda}^{i,(n-1)}_{\pm}:=\lambda^{i}_{\pm}(z^{i,(n-1)}_{-},z^{i,(n-1)}_{+}),
\end{aligned}
\end{eqnarray}
for $i=a, b$,
\begin{eqnarray}\label{eq:3.48}
\begin{aligned}
&\alpha^{(n-1)}:=\alpha(z^{a,(n-1)};\widetilde{S}^{a}_{0},\widetilde{B}^{a}_{0})
=\frac{1}{2\int^{1}_{0}\partial_{\tau}\Theta\big(\underline{p}^{a}+\tau(p^{a, (n-1)}-\underline{p});\widetilde{S}^{a}_{0},\widetilde{B}^{a}_{0}\big)d\tau},\\
&\beta^{(n-1)}:=\beta(z^{b,(n-1)};\widetilde{S}^{b}_{0},\widetilde{B}^{b}_{0})
=\frac{1}{2\int^{1}_{0}\partial_{\tau}\Theta\big(\underline{p}^{b}+\tau(p^{b, (n-1)}-\underline{p}^{b});\widetilde{S}^{b}_{0},\widetilde{B}^{b}_{0}\big)d\tau},
\end{aligned}
\end{eqnarray}
and
\begin{eqnarray}\label{eq:3.49}
\begin{aligned}
c(\xi)= p(\underline{z}^{a}; \underline{S}^{a},\underline{B}^{a})
-p(\underline{z}^{a}; \widetilde{S}^{a}_{0},\widetilde{B}^{a}_{0})
+p(\underline{z}^{b}; \widetilde{S}^{b}_{0},\widetilde{B}^{b}_{0})
-p(\underline{z}^{b}; \underline{S}^{b},\underline{B}^{b}).
\end{aligned}
\end{eqnarray}

\par 
Define the iteration set as
\begin{eqnarray*}
\mathcal{M}_{\epsilon}=\big\{(z^{a,(n)}, z^{b,(n)}):\ \| z^{a,(n)}-\underline{z}^{a}\|_{C^{2}(\widetilde{\Omega}^{(a)})}
+\| z^{b,(n)}-\underline{z}^{b}\|_{C^{2}(\widetilde{\Omega}^{(b)})}\leq \epsilon \big\},
\end{eqnarray*}
for some $0<\epsilon<1$.
\par Next, let us introduce the map
\begin{eqnarray}\label{eq:3.51}
\mathcal{T}:\  \mathcal{M}_{2\sigma}\longrightarrow \mathcal{M}_{2\sigma}.
\end{eqnarray}
For given functions $(z^{a,(n-1)}, z^{b,(n-1)})\in \mathcal{M}_{2\sigma}$,
we solve the linearized  boundary value problem $(\mathbf{\widetilde{P}_{n}})$ to obtain functions $(z^{a,(n)}, z^{b,(n)})\in \mathcal{M}_{2\sigma}$.
Then the map $\mathcal{T}$ is defined such that
\begin{eqnarray}\label{eq:3.52}
(z^{a,(n)}, z^{b,(n)}):=\mathcal{T}(z^{a,(n-1)}, z^{b,(n-1)}).
\end{eqnarray}
For the map $\mathcal{T}$, we shall show the following two facts: 

\medskip
\par (1)\ $\mathcal{T}$ is well-defined, \emph{i.e.}, $\mathcal{T}$ exists and maps from $\mathcal{M}_{2\sigma}$ to itself;
\par (2)\ $\mathcal{T}$ is a contraction map. 

\medskip
\noindent
Once the above facts are proved, the convergence follows from the Banach fixed point theorem and one can show that the limit  is  the solution
of the boundary value problem $(\widetilde{\mathbf{P}})$. The first fact will be proved in Theorem \ref{thm:4.1} and the second fact will be proved in Proposition \ref{prop:5.1}.
\bigskip

\section{Estimates for the Solutions to the Problem $(\mathbf{\widetilde{P}_{n}})$
in  $\widetilde{\Omega}^{(a)}\cup\widetilde{\Omega}^{(b)}$}\setcounter{equation}{0}

In this section, we shall consider the solutions 
of Problem $(\mathbf{\widetilde{P}_{n}})$ near the background state $\underline{z}$ and prove the following theorem.

\begin{theorem}\label{thm:4.1}
There exist positive constants $C^{*}$ and $\sigma^{*}_{0}$  depending only on $ \widetilde{\underline{U}}$, $L$ and $\gamma$
such that, for any $\sigma\in (0,\sigma^{*}_{0})$,  if $(\delta z^{a,(n-1)},\delta z^{b,(n-1)})\in \mathcal{M}_{2\sigma}$,
then the solutions $(\delta z^{a,(n)},\delta z^{b,(n)})$ to the problem $(\mathbf{\widetilde{P}_{n}})$ satisfy
\begin{eqnarray}\label{eq:4.71}
\begin{aligned}
&\|\delta z^{a,(n)}\|_{C^{2}(\Omega^{(a)})}+\|\delta z^{b,(n)}\|_{C^{2}(\Omega^{(b)})}\\
&\ \ \ \leq C^{*}\Big(\sum_{i=a,b}\big(\|z^{i}_{0}-\underline{z}^{i}\|_{C^{2}(\Sigma_{i})}
+\|\widetilde{S}^{i}_{0}-\underline{S}^{i}\|_{C^{2}( \Sigma_{i})}\big)\\
&\ \ \ \ \ \  +\sum_{i=a,b}\|\widetilde{B}^{i}_{0}-\underline{B}^{i}\|_{C^{2}( \Sigma_{i})}
+\|g_{+}-1\|_{C^{3}([0, L])}+\|g_{-}+1\|_{C^{3}([0, L])}\Big).
\end{aligned}
\end{eqnarray}
where $\Sigma_{a}=(0, m^{(a)}),\ \Sigma_{b}=(-m^{(b)}, 0)$.
\end{theorem}


We shall divide the proof of the Theorem \ref{thm:4.1} into several parts for different subdomains.

 \vspace{2pt}
\begin{figure}[ht]
\begin{center}
\begin{tikzpicture}[scale=1.5]
\draw [line width=0.06cm] (-5,1) --(2,1);
\draw [line width=0.06cm] (-5,-2) --(2,-2);
\draw [thick](-5,-2) --(-5,1);
\draw [thick][red](-5,-0.5) --(-2.5,1);
\draw [thick][red](-5,-2) --(0,1);
\draw [thick][red](-2.5,1)--(1.5,-1.5) ;
\draw [thick][blue](-5,-0.5) --(-2.5,-2);
\draw [thick][blue](-5,1) --(0,-2);
\draw [thick][blue](-2.5,-2)--(1.5,0.5) ;
\node at (-5.5, -0.5) {$(0,0)$};
\node at (2.6, 1) {$\eta=m^{(a)}$};
\node at (2.6, -2) {$\eta=-m^{(b)}$};
\node at (-5, -2.4) {$\xi=0$};
\node at (-4.0, -0.2) {$l^{0}_{+}$};
\node at (-4.0, -0.8) {$l^{0}_{-}$};
\node at (-3.0, -1.1) {$l^{1}_{+}$};
\node at (-3.0, 0.1) {$l^{1}_{-}$};
\node at (-4.5, 0.2) {$\Omega_{ I}$};
\node at (-3.7, 0.6) {$\Omega_{ II}$};
\node at (-2.5, 0.3) {$\Omega_{ IV}$};
\node at (-3.4,-0.6) {$\Omega_{III}$};
\node at (-4.5, -1.3) {$\Omega_{ I}$};
\node at (-3.7, -1.7) {$\Omega_{ II}$};
\node at (-2.5, -1.3) {$\Omega_{ IV}$};
\end{tikzpicture}
\caption{ }\label{fig:4.1}
\end{center}
\end{figure}

As shown in Fig. \ref{fig:4.1}, the subdomains are determined as follows. Let $\Omega_I$
be the union of two triangles which are bounded by the inlet $\xi=0$, the two characteristics
of \eqref{eq:4.1a} corresponding to $\lambda_+$ starting from point $(0,-m^{(b)})$
(denoted by $l^1_+$) and from point $(0,0)$ (denoted by $l^0_+$),
and the other two characteristics of \eqref{eq:4.1a} corresponding to $\lambda_-$
starting from point $(0,m^{(a)})$ (denoted by $l_-^1$) and from point $(0,0)$
(denoted by $l_-^0$).
Let $\Omega_{II}$ be the union of two triangles bounded by the nozzle walls $\eta=m^{(a)}$
(or $\eta=-m^{(b)}$),$l_+^0$ (or $l^0_-$), and $l_-^1$ (or $l_+^1$).
Let $\Omega_{III}$ be the diamond bounded by $l^0_{\pm}$ and $l^1_{\pm}$.
Let $\Omega_{IV}$ be the union of two diamonds bounded by $l^0_{\pm}$, $l^1_{\pm}$,
the characteristic corresponding to $\lambda_{+}$ starting from the intersection point
of $l_-^0$ and $\eta=-m^{(b)}$, and the characteristic corresponding to $\lambda_-$
starting from the intersection point of $l_+^0$ and $\eta=m^{(a)}$.

\subsection{Estimates of solutions of the boundary value problem  $(\mathbf{\widetilde{P}_{n}})$
in $\Omega_I$}
In this subsection, we study the problem $(\mathbf{\widetilde{P}_{n}})$
in $\Omega_I$ by the characteristic method.
Notice that $\Omega_I$ is bounded by the inlet and the characteristics issuing from the inlet,
so if we regard $\xi$ as the time, then the problem $(\mathbf{\widetilde{P}_{n}})$
in $\Omega_I$ can be regarded as the following initial value problem $(\mathbf{\widetilde{P}_{n}})_1$:
\begin{eqnarray}\label{eq:4.1}
\begin{aligned}
(\mathbf{\widetilde{P}_{n}})_1\quad \left\{
\begin{array}{llll}
\partial_{\xi}\delta z^{a,(n)}+diag(\tilde{\lambda}^{a,(n-1)}_{+},
\tilde{\lambda}^{a,(n-1)}_{-})\partial_{\eta}\delta z^{a,(n)}=0,
&\ \  \ in\quad \widetilde{\Omega}^{(a)}, \\
\partial_{\xi}\delta z^{b,(n)}+diag(\tilde{\lambda}^{b,(n-1)}_{+},
\tilde{\lambda}^{b,(n-1)}_{-})\partial_{\eta}\delta z^{b,(n)}=0,
&\ \  \ in\quad \widetilde{\Omega}^{(b)}, \\
\delta z^{a,(n)}=\delta z^{(a)}_{0}, &\ \ \ on\quad  \xi=0,\\
\delta z^{b,(n)}=\delta z^{(b)}_{0}, &\ \ \ on\quad  \xi=0.
\end{array}
\right.
\end{aligned}
\end{eqnarray}

\par 
For $\eta_{0}\in [-m^{(b)}, m^{(a)}]$, let $\eta=\chi^{i,(n)}_{\pm}(\xi,\eta_{0})$
be the characteristic curves corresponding to $\tilde{\lambda}^{i,(n-1)}_{\pm}$,
issuing from the point $(0,\eta_{0})$, \emph{i.e.}, for $i=a$ or $b$,
\begin{eqnarray*}
\begin{aligned}
\left\{
\begin{array}{llll}
\frac{d\chi^{i,(n)}_{\pm}}{d\xi}=\tilde{\lambda}^{i,(n-1)}_{\pm}(\xi,\chi^{i,(n)}_{\pm}),\\
\chi^{i,(n)}_{\pm}(0,\eta_{0})=\eta_{0}.
\end{array}
\right.
\end{aligned}
\end{eqnarray*}
Then,  along the characteristics one has
\begin{eqnarray}\label{eq:4.3}
\eta=\eta_{0}+\int^{\xi}_{0}\tilde{\lambda}^{i,(n-1)}_{\pm}(\tau,\chi^{i,(n)}_{\pm}(\tau, \eta_{0}))d\tau.
\end{eqnarray}
By \eqref{eq:4.3}, for any given point $(\xi,\eta)$ in $\Omega_I$, there exist unique $\eta_0^{\pm}$ such that the characteristics corresponding to $\lambda_{\pm}$ issuing from $(0,\eta_0^{\pm})$ respectively pass through $(\xi,\eta)$. Thus we can regard $\eta_0^{\pm}$ as a function of $(\xi,\eta)$ in $\Omega_I$. For simplicity of notation, we write $\eta_0^{\pm}$ as $\eta_0$.

We have the following estimates for $\eta_{0}$.
\begin{lemma}\label{lem:4.1}
For any $\delta z^{i,(n-1)}\in \mathcal{M}_{2\sigma}$ $(i=a, b)$, there exist constants $C_{0,1}$ and $C_{0,2}$
 depending only on $ \widetilde{\underline{U}}$, $L$ and $\sigma$ such that
\begin{eqnarray}\label{eq:4.4}
\begin{aligned}
\|D^{k}\eta_{0}\|_{C^{0}(\tilde{\Omega}_a\cup\tilde{\Omega}_b)}\leq C_{0,k},\ \ \  k=1,2.
\end{aligned}
\end{eqnarray}
\end{lemma}

\begin{proof}
Without loss of the generality, we only consider the estimates of \eqref{eq:4.4} in $\Omega_I$, 
since otherwise we should consider the reflection of the characteristics by the nozzle walls which can be uniformly bounded by a constant depending only on the background solution. The discussion of the reflection of the characteristics by the nozzle walls is postponed to Lemma \ref{lem:4.2}.

Taking derivatives on \eqref{eq:4.3} with respect to $\xi, \eta$ and by the direct computation we can obtain the formulas of the first and second order partial derivatives of $\eta_0$ with respect to $\xi, \eta$, for example,
$$\frac{\partial \eta_{0}}{\partial\eta}=\frac{1}
{1+\int^{\xi}_{0}\partial_{\eta}\tilde{\lambda}^{i,(n-1)}_{\pm}
e^{\int^{\tau}_{0}\partial_{\eta}\tilde{\lambda}^{i,(n-1)}_{\pm}ds}d\tau},\ \ \  i=a, b,
$$
and
$$\frac{\partial^{2}\eta_{0}}{\partial\eta^{2}}
=-\frac{\int^{\xi}_{0}\partial_{\eta}\Big(\partial_{\eta}\tilde{\lambda}^{i,(n-1)}_{\pm}
e^{\int^{\tau}_{0}\partial_{\eta}\tilde{\lambda}^{i,(n-1)}_{\pm}ds}\Big)
e^{\int^{\tau}_{0}\partial_{\eta}\tilde{\lambda}^{i,(n-1)}_{\pm}ds} d\tau \frac{\partial \eta_{0}}{\partial \eta}}
{\Big(1+\int^{\xi}_{0}\partial_{\eta}\tilde{\lambda}^{i,(n-1)}_{\pm}
e^{\int^{\tau}_{0}\partial_{\eta}\tilde{\lambda}^{i,(n-1)}_{\pm}ds}d\tau\Big)^{2}}.
$$
Other derivatives have the similar formulas.
From these formulas we can deduce the estimates for $\|D^{2}\eta_{0}\|_{C^{0}(\Omega_I)}$.
This complete the proof.
\end{proof}


Let $\xi_1^{a,+}$ be the intersection point of the characteristics corresponding to $\lambda_{+}$ starting from $(0,0)$ and the upper nozzle wall $\eta=m^{(a)}$. Let $\xi_1^{b,-}$ be the intersection point of the characteristics corresponding to $\lambda_-$ starting from $(0,0)$ and the lower nozzle wall $\eta=-m^{(b)}$,  
\emph{i.e.}, it holds that
\begin{eqnarray}\label{eq:4.7}
\begin{aligned}
&\int^{\xi^{a,+}_{1}}_{0}\tilde{\lambda}^{a,(n-1)}_{+}(\tau,\chi^{a,(n)}_{+,1}(\tau,0))d\tau=m^{(a)},\\
&\int^{\xi^{b,-}_{1}}_{0}\tilde{\lambda}^{b,(n-1)}_{-}(\tau,\chi^{b,(n)}_{-,1}(\tau, 0))d\tau=-m^{(b)}.
\end{aligned}
\end{eqnarray}

Then we have the following proposition.
\begin{proposition}\label{prop:4.1}
There exist positive constants $\widetilde{C}^{i}_{1}$ 
and $\sigma_{0}$
depending only on $\underline{\widetilde{U}}$ and $L,\gamma$ such that,  for any $\sigma\in (0,\sigma_{0})$,
if $\delta z^{i,(n-1)}\in \mathcal{M}_{2\sigma}$, the solution $\delta z^{i,(n)}$ to
the initial value problem $(\mathbf{\widetilde{P}_{n}})_{1}$ satisfies
\begin{eqnarray}\label{eq:4.9}
\begin{aligned}
\|\delta z^{a,(n)}_{-}\|_{C^{2}(\widetilde{\Omega}^{(a)}\cap(\Omega_I\cup\Omega_{II}))}&+\|\delta z^{b,(n)}_{-}\|_{C^{2}(\widetilde{\Omega}^{(b)}\cap(\Omega_I\cup\Omega_{III}))}\\
&\ \ \  \leq\widetilde{C}^{i}_{1}\Big(\|\delta z^{(a)}_{-,0}\|_{C^{2}([0,m^{(a)}])}
+\|\delta z^{(b)}_{-,0}\|_{C^{2}([-m^{(b)},0])}\Big),
\end{aligned}
\end{eqnarray}
and
\begin{eqnarray}\label{eq:4.10}
\begin{aligned}
\|\delta z^{a,(n)}_{+}\|_{C^{2}(\widetilde{\Omega}^{(a)}\cap(\Omega_I\cup\Omega_{III}))}&+\|\delta z^{b,(n)}_{+}\|_{C^{2}(\widetilde{\Omega}^{(b)}\cap(\Omega_I\cup\Omega_{II}))}\\
& \ \ \ \leq\widetilde{C}^{i}_{1}\Big(\|\delta z^{(a)}_{+,0}\|_{C^{2}([0,m^{(a)}])}+\|\delta z^{(b)}_{+,0}\|_{C^{2}([-m^{(b)},0])}\Big).
\end{aligned}
\end{eqnarray}
\end{proposition}
\begin{proof}
We only consider the estimates of the solution for $z^{a, (n)}_{-}$ in $\widetilde{\Omega}^{(a)}\cap(\Omega_I\cup\Omega_{II})$ since the proof
for the others are the same as $z^{a, (n)}_{-}$.
The proof is divided into three steps.

\medskip
\emph{Step 1}.
For any point $(0, \eta_{0})$ with $\eta_{0}\in[0, m^{(a)}]$, along the characteristic
$\eta=\chi^{a,(n)}_{+}(\xi,\eta_{0})$, we have $\delta z^{a,(n)}_{-}(\xi, \eta)=\delta z^{a}_{-,0}(\eta_{0})$.
Therefore, we have
\begin{eqnarray}\label{eq:4.11}
\|\delta z^{a,(n)}_{-}\|_{C^{0}(\widetilde{\Omega}^{(a)}\cap(\Omega_I\cup\Omega_{II}))}=\|\delta z_{-,0}\|_{C^{0}([0, m^{(a)}])}.
\end{eqnarray}

\medskip
\emph{Step 2}. We now work on the estimate of $\|D\delta z^{a,(n)}\|_{C^{0}(\widetilde{\Omega}^{(a)}\cap(\Omega_I\cup\Omega_{II}))}$.
To do this, we differentiate the equation $\eqref{eq:4.1}_{1}$ for $\delta z^{a, (n)}_{-}$
with respect to $\xi, \eta$ to obtain
\begin{eqnarray}\label{eq:4.12}
\partial_{\xi}( \partial_{\xi}\delta z^{a,(n)}_{-})+\tilde{\lambda}^{a,(n-1)}_{+}
\partial_{\eta}(\partial_{\xi}\delta z^{a,(n)}_{-})
=-\partial_{\xi}\tilde{\lambda}^{a,(n-1)}_{+}\partial_{\eta} \delta z^{a,(n)}_{-},
\end{eqnarray}
and
\begin{eqnarray}\label{eq:4.13}
\partial_{\xi}( \partial_{\eta}\delta z^{a,(n)}_{-})
+\tilde{\lambda}^{a,(n-1)}_{+}\partial_{\eta}(\partial_{\eta}\delta z^{a,(n)}_{-})
=-\partial_{\eta}\tilde{\lambda}^{a,(n-1)}_{+}\partial_{\eta}\delta z^{a,(n)}_{-}.
\end{eqnarray}
Then, along the $\eta=\chi^{a,(n)}_{+}(\xi,\eta_{0})$ characteristic, it holds that
\begin{eqnarray}\label{eq:4.14}
\begin{aligned}
\partial_{\xi} \delta z^{a,(n)}_{-}(\xi, \eta)&=\partial_{\xi}\delta z^{a,(n)}_{-}(0, \eta_{0})
-\int^{\xi}_{0}\big(\partial_{\tau}\tilde{\lambda}^{a,(n-1)}_{+}\partial_{\eta}\delta z^{a,(n)}_{-}\big)(\tau, \chi^{a,(n)}_+(\tau,\eta_0))d\tau\\
&=\frac{\partial \eta_{0}}{\partial \xi}\partial_{\eta_{0}}\delta z^{a}_{-,0}(\eta_0)
-\int^{\xi}_{0}\big(\partial_{\tau}\tilde{\lambda}^{a,(n-1)}_{+}\partial_{\eta}\delta z^{a,(n)}_{-}\big)(\tau; \eta_0)d\tau,
\end{aligned}
\end{eqnarray}
and
\begin{eqnarray}\label{eq:4.15}
\begin{aligned}
\partial_{\eta} \delta z^{a,(n)}_{-}(\xi, \eta)&=\partial_{\eta}\delta z^{a,(n)}_{-}(0, \eta_{0})
-\int^{\xi}_{0}\big(\partial_{\eta}\tilde{\lambda}^{a,(n-1)}_{+}\partial_{\eta}\delta z^{a,(n)}_{-}\big)(\tau,  \chi^{a,(n)}_+(\tau,\eta_0))d\tau\\
&=\frac{\partial \eta_{0}}{\partial\eta}\partial_{\eta_{0}}\delta z^{a}_{-,0}(\eta_0)
-\int^{\xi}_{0}\big(\partial_{\eta}\tilde{\lambda}^{a,(n-1)}_{+}\partial_{\eta}\delta z^{a,(n)}_{-}\big)(\tau; \eta_0)d\tau.
\end{aligned}
\end{eqnarray}

Now, we need to estimate each term in \eqref{eq:4.14} and \eqref{eq:4.15}.

First, let us consider the estimates for $\int^{\xi}_{0}\big(\partial_{\tau}\tilde{\lambda}^{a,(n-1)}_{+}\partial_{\eta}\delta z^{a,(n)}_{-}\big)(\tau, \eta_0)d\tau$ and\\  $\int^{\xi}_{0}\big(\partial_{\eta}\tilde{\lambda}^{a,(n-1)}_{+}\partial_{\eta}\delta z^{a,(n)}_{-}\big)(\tau, \eta_0)d\tau$.
Note that
\begin{eqnarray}\label{eq:4.16}
\begin{aligned}
& \int^{\xi}_{0}|\big(\partial_{\tau}\tilde{\lambda}^{a,(n-1)}_{+}\partial_{\eta}\delta z^{a,(n)}_{-}\big)
(\tau;\eta_0)|
d\tau
+\int^{\xi}_{0}|\big(\partial_{\eta}\tilde{\lambda}^{a,(n-1)}_{+}
\partial_{\eta}\delta z^{a,(n)}_{-}\big)(\tau; \eta_0)|
d\tau\\
&\ \  \leq \mathcal{C}\|D\delta z^{a,(n-1)}\|_{C^{0}(\widetilde{\Omega}^{a})}\int^{\xi}_{0}
|D\delta z^{a,(n)}_{-}(\tau; \eta_0)| 
d\tau\\
&\ \  \leq 2\mathcal{C}\sigma\int^{\xi}_{0}
|D\delta z^{a,(n)}_{-}(\tau;\eta_0)|
d\tau,
\end{aligned}
\end{eqnarray}
where $\mathcal{C}$ is a positive constant that depends only
on the background solution $ \widetilde{\underline{U}}$ and $\sigma$.
Hence $\mathcal{C}$ essentially depends only on $ \widetilde{\underline{U}}$
if $\sigma$ is small.
%
Then by \eqref{eq:4.16}
and Lemma \ref{lem:4.1}, we have
\begin{eqnarray*}
\begin{aligned}
|D\delta z^{a,(n)}_{-} (\xi,\chi^{a,(n)}_{+}(\xi,\eta_0))|
& \leq \|D\eta_{0}\|_{C^{0}(\tilde{\Omega}_a)}
\| \partial_{\eta_{0}}\delta z^{a}_{-,0}\|_{C^{0}([0,m^{(a)}])}\\
&\ \ \ \  + \int^{\xi}_{0}|\big(\partial_{\tau}\tilde{\lambda}^{a,(n-1)}_{+}\partial_{\eta}\delta z^{a,(n)}_{-}\big)
(\tau; \eta_0)|
d\tau\\
&\ \ \ \  +\int^{\xi}_{0}|\big(\partial_{\eta}\tilde{\lambda}^{a,(n-1)}_{+}
\partial_{\eta}\delta z^{a,(n)}_{-}\big)(\tau;\eta_0)|
d\tau\\
&\leq  \mathcal{C}\| \partial_{\eta_{0}}\delta z^{a}_{-,0}\|_{C^{0}([0,m^{(a)}])}
+2\mathcal{C}\sigma\int^{\xi}_{0}
|D\delta z^{a,(n)}_{-}(\tau;\eta_0)|
d\tau,
\end{aligned}
\end{eqnarray*}
where $\mathcal{C}$ is a positive constant
that depends only on the background solution $ \widetilde{\underline{U}}$ and $L, \gamma$
for $\sigma$ sufficiently small.
 Applying the Gronwall inequality  yields
\begin{eqnarray*}
\begin{aligned}
\|D\delta z^{a,(n)}_{-}\|_{C^{0}(\widetilde{\Omega}^{(a)}\cap(\Omega_I\cup\Omega_{II}))}
& \leq  \mathcal{C}\| \partial_{\eta_{0}}\delta z^{a}_{-,0}\|_{C^{0}([0,m^{(a)}])}
e^{2\mathcal{C}L\sigma}.  
\end{aligned}
\end{eqnarray*}
Choose $\sigma'_{0}>0$ such that $e^{2\mathcal{C}L\sigma}\leq 2$ for $\sigma \in (0, \sigma'_{0})$.
Then, it follows that
\begin{eqnarray}\label{eq:4.17}
\begin{aligned}
\|D\delta z^{a,(n)}_{-}\|_{C^{0}(\widetilde{\Omega}^{(a)}\cap(\Omega_I\cup\Omega_{II}))}
\leq  2\mathcal{C}\| \partial_{\eta_{0}}\delta z^{a}_{-,0}\|_{C^{0}([0,m^{(a)}])}.
\end{aligned}
\end{eqnarray}

\medskip
\par \emph{Step 3}. Finally, let us consider the estimate for $\|D^{2}\delta z^{a,(n)}_{-}\|_{C^{0}(\widetilde{\Omega}^{(a)}\cap(\Omega_I\cup\Omega_{II}))}$.
To do this, we differentiate equations \eqref{eq:4.12} \eqref{eq:4.13} with respect to
$(\xi, \eta)$ again to obtain
\begin{eqnarray}\label{eq:4.18}
\begin{aligned}
&\partial_{\xi}( \partial^2_{\xi\xi}\delta z^{a,(n)}_{-})+\tilde{\lambda}^{a,(n-1)}_{+}
\partial_{\eta}(\partial^{2}_{\xi\xi}\delta z^{a,(n)}_{-})\\
& \ \ \ \ =-2\partial_{\xi}\tilde{\lambda}^{a,(n-1)}_{+}\partial^{2}_{\xi\eta} \delta z^{a,(n)}_{-}
-\partial_{\xi\xi}\tilde{\lambda}^{a,(n-1)}_{+}\partial_{\eta} \delta z^{a,(n)}_{-},\\
&\partial_{\xi}( \partial^2_{\xi\eta}\delta z^{a,(n)}_{-})+\tilde{\lambda}^{a,(n-1)}_{+}
\partial_{\eta}(\partial^{2}_{\xi\eta}\delta z^{a,(n)}_{-})\\
& \ \ \ \ =-\partial_{\eta}\tilde{\lambda}^{a,(n-1)}_{+}\partial^{2}_{\xi\eta} \delta z^{a,(n)}_{-}
-\partial_{\xi}\tilde{\lambda}^{a,(n-1)}_{+}\partial^{2}_{\eta\eta} \delta z^{a,(n)}_{-}
-\partial_{\xi\eta}\tilde{\lambda}^{a,(n-1)}_{+}\partial_{\eta} \delta z^{a,(n)}_{-},
\end{aligned}
\end{eqnarray}
and
\begin{eqnarray}\label{eq:4.19}
\begin{aligned}
&\partial_{\xi}( \partial_{\eta\eta}\delta z^{a,(n)}_{-})
+\tilde{\lambda}^{a,(n-1)}_{+}\partial_{\eta}(\partial^{2}_{\eta\eta}\delta z^{a,(n)}_{-})\\
&\ \ \ \ =-2\partial_{\eta}\tilde{\lambda}^{a,(n-1)}_{+}\partial^{2}_{\eta\eta}\delta z^{a,(n)}_{-}
-\partial^{2}_{\eta\eta}\tilde{\lambda}^{a,(n-1)}_{+}\partial_{\eta} \delta z^{a,(n)}_{-}.
\end{aligned}
\end{eqnarray}
Then, integrating  equations \eqref{eq:4.18} and \eqref{eq:4.19} along the characteristic
$\eta=\chi^{a,(n)}_{+}(\xi,\eta_{0})$ from $0$ to $\xi$, we have
\begin{eqnarray}\label{eq:4.20}
\begin{aligned}
\partial^2_{\xi\xi}\delta z^{a,(n)}_{-}&=\partial^2_{\xi\xi}\delta z^{a,(n)}_{-}(0,\eta_0)
-2\int^{\xi}_{0}\partial_{\tau}\tilde{\lambda}^{a,(n-1)}_{+}\partial^{2}_{\tau\eta} \delta z^{a,(n)}_{-}d\tau\\
&\ \ \ \ -\int^{\xi}_{0}\partial^{2}_{\tau\tau}\tilde{\lambda}^{a,(n-1)}_{+}\partial_{\eta} \delta z^{a,(n)}_{-}d\tau,\\
&=\big(\frac{\partial\eta_{0}}{\partial\xi}\big)^{2}\partial^{2}_{\eta_{0}\eta_{0}}\delta z^{a}_{-,0}
+\frac{\partial^{2}\eta_{0}}{\partial\xi^{2}}\partial_{\eta_{0}}\delta z^{a}_{-,0}
-2\int^{\xi}_{0}\partial_{\tau}\tilde{\lambda}^{a,(n-1)}_{+}\partial^{2}_{\tau\eta} \delta z^{a,(n)}_{-}d\tau\\
&\ \ \ \ -\int^{\xi}_{0}\partial^{2}_{\tau\tau}\tilde{\lambda}^{a,(n-1)}_{+}\partial_{\eta} \delta z^{a,(n)}_{-}d\tau\\
&:=J_{11}+E_{11},
\end{aligned}
\end{eqnarray}
\begin{eqnarray}\label{eq:4.21}
\begin{aligned}
\partial^2_{\xi\eta}\delta z^{a,(n)}_{-}
&=\frac{\partial\eta_{0}}{\partial\xi}\frac{\partial\eta_{0}}{\partial\eta}\partial^{2}_{\eta_{0}\eta_{0}}\delta z^{a}_{-,0}
+\frac{\partial^{2}\eta_{0}}{\partial\xi\partial\eta}\partial_{\eta_{0}}\delta z^{a}_{-,0}
-\int^{\xi}_{0}\partial_{\eta}\tilde{\lambda}^{a,(n-1)}_{+}\partial^{2}_{\tau\eta} \delta z^{a,(n)}_{-}d\tau\\
&\ \ \ \ -\int^{\xi}_{0}\partial_{\xi}\tilde{\lambda}^{a,(n-1)}_{+}\partial^{2}_{\eta\eta} \delta z^{a,(n)}_{-}d\tau
-\int^{\xi}_{0}\partial^{2}_{\xi\eta}\tilde{\lambda}^{a,(n-1)}_{+}\partial_{\eta} \delta z^{a,(n)}_{-}d\tau\\
&:=J_{12}+E_{12},
\end{aligned}
\end{eqnarray}
and
\begin{eqnarray}\label{eq:4.22}
\begin{aligned}
\partial^2_{\eta\eta}\delta z^{a,(n)}_{-}
&=\Big(\frac{\partial\eta_{0}}{\partial\eta}\Big)^{2}\partial^{2}_{\eta_{0}\eta_{0}}\delta z^{a}_{-,0}
+\frac{\partial^{2}\eta_{0}}{\partial\eta^{2}}\partial_{\eta_{0}}\delta z^{a}_{-,0}\\
&\ \ \ \ -2\int^{\xi}_{0}\partial_{\eta}\tilde{\lambda}^{a,(n-1)}_{+}\partial^{2}_{\eta\eta} \delta z^{a,(n)}_{-}d\tau -\int^{\xi}_{0}\partial^{2}_{\eta\eta}\tilde{\lambda}^{a,(n-1)}_{+}\partial_{\eta} \delta z^{a,(n)}_{-}d\tau\\
&:=J_{22}+E_{22}.
\end{aligned}
\end{eqnarray}
Here $J_{11}, J_{12}, J_{22}$ denote the first two terms on the right hands of \eqref{eq:4.20}-\eqref{eq:4.22}
and $E_{11}, E_{12}, E_{22}$ represent the rest of them.

For $J_{11},\ J_{12},\ J_{22}$, we have, using Lemma \ref{lem:4.1},
\begin{eqnarray}\label{eq:4.23a}
\begin{aligned}
|J_{11}+J_{12}+J_{22}|
&
\leq\mathcal{C}\big(\|D\eta_{0}\|^{2}_{C^{0}(\widetilde{\Omega}_a)}
+\|D^{2}\eta_{0}\|_{C^{0}(\widetilde{\Omega}_a)}\big)
\|\delta z^{a}_{-,0}\|_{C^{2}([0, m^{(a)}])}\\
&
\leq \mathcal{C}\|\delta z^{a}_{-,0}\|_{C^{2}([0, m^{(a)}])},
\end{aligned}
\end{eqnarray}
where $\mathcal{C}$ depends only on $\widetilde{\underline{U}}$ and $L, \gamma$.

Next, for $E_{11},\ E_{12},\ E_{22}$, we have
\begin{eqnarray}\label{eq:4.24}
\begin{aligned}
&|E_{11}+E_{12}+E_{22}|
\\
&  \leq\mathcal{C}\|\delta z^{a,(n-1)}\|_{C^{2}(\widetilde{\Omega}^{a})}
\int^{\xi}_{0}|D^{2}\delta z^{a,(n)}_{-}(\tau;\eta_0)|
d\tau
+\mathcal{C}\|\delta z^{a,(n-1)}\|_{C^{2}(\widetilde{\Omega}^{a})}
\int^{\xi}_{0}|D\delta z^{a,(n)}_{-}(\tau;\eta_0)|
d\tau\\
&  \leq\mathcal{C}\|\delta z^{a}_{-,0}\|_{C^{2}([0, m^{(a)}])}+2\mathcal{C}\sigma\int^{\xi}_{0}|D^{2}\delta z^{a,(n)}_{-}(\tau;\eta_0)|
d\tau,
\end{aligned}
\end{eqnarray}
where the positive constant $\mathcal{C}$ depends only on $ \widetilde{\underline{U}}$ and $L,\gamma$,
using the estimate \eqref{eq:4.17} and taking $0<\sigma<1$ sufficiently small.

Combining \eqref{eq:4.23a} and \eqref{eq:4.24} together,  we have
\begin{eqnarray*}
\begin{aligned}
|D^{2}\delta z^{a,(n)}_{-}(\xi,\chi_+^{a,(n)}(\xi,\eta_0))|
& \leq |\sum _{i,j=1,2}J_{ij}|
+|\sum _{i,j=1,2}E_{ij}|
\\
& \leq\mathcal{C}\|\delta z^{a}_{-,0}\|_{C^{2}([0, m^{(a)}])}+2\mathcal{C}\sigma\int^{\xi}_{0}|D^{2}\delta z^{a,(n)}_{-}(\tau;\eta_0)|
d\tau.
\end{aligned}
\end{eqnarray*}
Then by the Gronwall inequality, we have
\begin{eqnarray*}
\begin{aligned}
|D^{2}\delta z^{a,(n)}_{-}(\xi,\chi_+^{a,(n)}(\xi,\eta_0))|
 \leq \mathcal{C}\|\delta z^{a}_{-,0}\|_{C^{2}([0, m^{(a)}])}
e^{2\mathcal{C}\sigma\xi}.
\end{aligned}
\end{eqnarray*}
Hence, we can choose $\sigma''_{0}>0$ depending only on $\widetilde{\underline{U}}$ and $L,\gamma$
small enough such that for $\sigma \in (0, \sigma''_{0})$, it holds that
\begin{eqnarray}\label{eq:4.23}
\begin{aligned}
\|D^{2}\delta z^{a,(n)}_{-}\|_{C^{0}(\widetilde{\Omega}^{(a)}\cap(\Omega_I\cup\Omega_{II}))}
\leq  2\mathcal{C}\| \delta z^{a}_{-,0}\|_{C^{2}([0,m^{(a)}])}.
\end{aligned}
\end{eqnarray}
Finally,  from the  estimates \eqref{eq:4.12}, \eqref{eq:4.17} and \eqref{eq:4.23} together,   there
exist constants $\widetilde{C}^{a}_{+,1}>0$ and $\sigma_{0}=\min\{\sigma'_{0},\sigma''_{0}\}$
depending only on $\widetilde{\underline{U}}$ and $L$ such that for $\sigma\in (0,\sigma_{0})$, it holds that
\begin{eqnarray*}
\begin{aligned}
\|\delta z^{a,(n)}_{-}\|_{C^{2}(\widetilde{\Omega}^{(a)}\cap(\Omega_I\cup\Omega_{II}))}
\leq  \widetilde{C}^{a}_{1,+}\| \delta z^{a}_{-,0}\|_{C^{2}([0,m^{(a)}])}.
\end{aligned}
\end{eqnarray*}
This completes the proof of the proposition.
\end{proof}

\subsection{Estimates for the boundary value problem $(\mathbf{\widetilde{P}_{n}})$
in $\Omega_{II}$}

In this subsection, we shall analyze the boundary value problem $(\mathbf{\widetilde{P}_{n}})$
with the boundary conditions on $\widetilde{\Gamma}_{-}$ and $\widetilde{\Gamma}_{+}$
respectively. Note that the mathematical problem $(\mathbf{\widetilde{P}_{n}})$ in $\Omega_{II}$ can be reformulated as the following problem:
\begin{eqnarray}\label{eq:4.26}
\begin{aligned}
(\mathbf{\widetilde{P}_{n}})_{2}\quad \left\{
\begin{array}{llll}
\partial_{\xi}\delta z^{a,(n)}_{+}+\tilde{\lambda}^{a,(n-1)}_{-}\partial_{\eta}\delta z^{a,(n)}_{+}=0,
&\ \  \ in\quad \widetilde{\Omega}^{(a)}, \\
\partial_{\xi}\delta z^{b,(n)}_{-}+\tilde{\lambda}^{b,(n-1)}_{+}\partial_{\eta}\delta z^{b,(n)}_{-}=0,
&\ \  \ in\quad \widetilde{\Omega}^{(b)}, \\
\delta z^{a,(n)}_{-}+\delta z^{a,(n)}_{+}=2\arctan g'_{+},  &\ \ \ on\quad\widetilde{ \Gamma}_{+},\\
\delta z^{b,(n)}_{-}+\delta z^{b,(n)}_{+}=2\arctan g'_{-},  &\ \ \ on\quad\widetilde{ \Gamma}_{-},  \\
\delta z^{a,(n)}=\delta z^{(a)}_{0}, &\ \ \ on\quad  \xi=0,\\
\delta z^{b,(n)}=\delta z^{(b)}_{0}, &\ \ \ on\quad  \xi=0.
\end{array}
\right.
\end{aligned}
\end{eqnarray}
\par Similarly to \S 4.1,  we shall study  the problem $(\mathbf{\widetilde{P}_{n}})_{2}$ also
via the characteristics method but different from the one in \S 4.1 due to the reflection of the characteristics by the nozzle walls.

Let $\xi=\psi^{i,(n)}_{-}(\eta,\xi^{a})$ (or $\xi=\psi^{i,(n)}_{+}(\eta,\xi^{b})$)
be the characteristic curves corresponding to $\lambda_-$ ( or $\lambda_+$) and passing
through  the points $(\xi^{a}, m^{(a)})$ ( or $(\xi^{b}, -m^{(b)})$), {\emph{i.e.},}  defined by
\begin{eqnarray}\label{eq:4.27}
\begin{aligned}
\left\{
\begin{array}{llll}
\frac{d\psi^{a,(n)}_{-}}{d\eta}=\frac{1}{\tilde{\lambda}^{a,(n-1)}_{-}(\eta,\psi^{a,(n)}_{-})},\\
\psi^{a,(n)}_{-}(m^{(a)},\xi^{a})=\xi^{a},
\end{array}
\right.
\end{aligned}
\end{eqnarray}
and
\begin{eqnarray}\label{eq:4.28}
\begin{aligned}
\left\{
\begin{array}{llll}
\frac{d\psi^{b,(n)}_{+}}{d\eta}=\frac{1}{\tilde{\lambda}^{b,(n-1)}_{+}(\eta,\psi^{b,(n)}_{+})},\\
\psi^{a,(n)}_{+}(-m^{(b)},\xi^{b})=\xi^{b},
\end{array}
\right.
\end{aligned}
\end{eqnarray}
where $\xi^{a},\ \xi^{b}\in [0, L]$.
 Then,
\begin{eqnarray*}
\begin{aligned}
&\xi=\xi^{a}+\int^{\eta}_{m^{(a)}}\frac{1}{\tilde{\lambda}^{a,(n-1)}_{-}(\psi^{a,(n)}_{-}(\tau,\xi^a),\tau)}d\tau,\ \ \
\xi=\xi^{b}+\int^{\eta}_{-m^{(b)}}\frac{1}{\tilde{\lambda}^{b,(n-1)}_{+}(\psi^{b,(n)}_{+}(\tau,\xi^b),\tau)}d\tau.
\end{aligned}
\end{eqnarray*}
For any point $(\xi,\eta)$ in ${\widetilde{\Omega}^{(a)}\cap\Omega_{II}}$ (or in ${\widetilde{\Omega}^{(b)}\cap\Omega_{II}}$),
there exists a unique $\xi^a$ (or $\xi^b$) such that the characteristic corresponding to $\lambda_-$ (or $\lambda_+$) passes through $(\xi,\eta)$.
Thus we can consider $\xi^{a}$ (or $\xi^b$) as a function of $(\xi,\eta)$ in ${\widetilde{\Omega}^{(a)}\cap\Omega_{II}}$ (or in ${\widetilde{\Omega}^{(b)}\cap\Omega_{II}}$), and
  we have the following lemma for $\xi^{a}$ and $\xi^{b}$.
\begin{lemma}\label{lem:4.2}
For any $\delta z^{i,(n-1)}\in \mathcal{M}_{2\sigma}$ $(i=a, b)$, there exist constants $C^{i}_{1,1}$ and $C^{i}_{1,2}$
 depending only on $\underline{\widetilde{U}}$, $L$ and $\sigma$ such that
\begin{eqnarray*}
\begin{aligned}
\|D^{k}\xi^{i}\|_{C^{0}(\widetilde{\Omega}^{(i)}\cap\Omega_{II})}\leq C^{i}_{1,k},\ \ \   k=1,2,
\end{aligned}
\end{eqnarray*}
where $i=a$, or $b$.
\end{lemma}

\begin{proof}
By the equation \eqref{eq:4.27}, the straightforward computation gives us the formulas of   the first order and the second order derivatives
of $\xi^{a}$ with respect to $\xi, \eta$, for example,
$$\frac{\partial \xi^{a}}{\partial \xi}=\frac{1}{1+\int^{\eta}_{m^{(a)}}
\partial_{\xi}(\frac{1}{\tilde{\lambda}^{a,(n-1)}_{-}})e^{\int^{\tau}_{m^{(a)}}
\partial_{\xi}(\frac{1}{\tilde{\lambda}^{a,(n-1)}_{-}}) ds} d\tau},$$
and
$$\frac{\partial^{2} \xi^{a}}{\partial \xi^{2}}=-\frac{\int^{\eta}_{m^{(a)}}
\partial_{\xi}\big(\partial_{\xi}(\frac{1}{\tilde{\lambda}^{a,(n-1)}_{-}})e^{\int^{\tau}_{m^{(a)}}
\partial_{\xi}(\frac{1}{\tilde{\lambda}^{a,(n-1)}_{-}}) ds}\big)e^{\int^{\tau}_{m^{(a)}}
\partial_{\xi}(\frac{1}{\tilde{\lambda}^{a,(n-1)}_{-}}) ds} d\tau \frac{\partial \xi^{a}}{\partial \xi}}
{\big(1+\int^{\eta}_{m^{(a)}}\partial_{\xi}(\frac{1}{\tilde{\lambda}^{a,(n-1)}_{-}})e^{\int^{\tau}_{m^{(a)}}
\partial_{\xi}(\frac{1}{\tilde{\lambda}^{a,(n-1)}_{-}}) ds} d\tau\big)^{2}},$$
and the other derivatives are similar.
Then the estimates for $D\xi^{a}$ and $D^{2}\xi^{a}$ follow from the above formulas.
In the same way, we can also get the estimates for $D\xi^{b}$ and $D^{2}\xi^{b}$.
\end{proof}

\par Our main result in this subsection is as follows.
\begin{proposition}\label{prop:4.2}
There exist constants $\widetilde{C}^{a}_{2,-}>0,\ \widetilde{C}^{b}_{2,+}>0$ and $\sigma_{1}>0$
depending only on $\underline{\widetilde{U}}$ and $L, \gamma$ such that for $\sigma\in (0,\sigma_{1})$,
if $\delta z^{i,(n-1)}\in \mathcal{M}_{2\sigma}$ and $(\delta z^{a,(n)}_{+}, \delta z^{b,(n)}_{-})$ are the solutions to
the boundary value problem $(\mathbf{\widetilde{P}_{n}})_{2}$, then
\begin{eqnarray}\label{eq:4.35}
\|\delta z^{a,(n)}_{+}\|_{C^{2}(\widetilde{\Omega}^{(a)}\cap\Omega_{II})}\leq
\widetilde{C}^{a}_{2,-}\Big(\|\delta z^{a, (n)}_{-}\|_{C^{2}([0,\xi_{1}^{a, +}]\times\{m^{(a)}\})}
+\|g_{+}-1\|_{C^{3}([0, \xi^{a,+}_{1}])}\Big),
\end{eqnarray}
and
\begin{eqnarray}\label{eq:4.36}
\|\delta z^{b,(n)}_{-}\|_{C^{2}(\widetilde{\Omega}^{(b)}\cap\Omega_{II})}\leq
\widetilde{C}^{b}_{2,+}\Big(\|\delta z^{b, (n)}_{+}\|_{C^{2}([0,\xi^{b,-}_1]\times \{-m^{(b)}\})}
+\|g_{-}+1\|_{C^{3}([0, \xi^{b,-}_{1}])}\Big),
\end{eqnarray}
where $\xi^{a,+}_{1}$ and $\xi^{b,-}_{1}$ are given in \eqref{eq:4.7}.
\end{proposition}
\begin{remark}
Together with Proposition \ref{prop:4.1}, we actually have the estimates of $\delta z_{\pm}^{i,(n)}$ in $\Omega_I\cup\Omega_{II}$ for $i=a$ or $b$.
\end{remark}

\begin{proof}
Similarly to the proof of Proposition \ref{prop:4.1}, without loss of the generality, we only consider the estimate of $\delta z^{a,(n)}_{+}$. 
By $\eqref{eq:4.26}_1$,
\begin{eqnarray}\label{eq:4.37}
\begin{aligned}
\partial_{\eta}\delta z^{a,(n)}_{+}+\frac{1}{\tilde{\lambda}^{a,(n-1)}_{-}}\partial_{\xi}\delta z^{a,(n)}_{+}=0.
\end{aligned}
\end{eqnarray}
Therefore along the characteristic  $\xi=\psi^{i,(n)}_{-}(\eta,\xi^{a})$, where $\xi^{a}\in [0,\xi^{a, +}_{1}]$,
we have
\begin{eqnarray}\label{eq:4.38}
\begin{aligned}
\|\delta z^{a,(n)}_{+}\|_{C^{0}(\widetilde{\Omega}^{(a)}\cap\Omega_{II})}\leq
\|\delta z^{a,(n)}_{-}\|_{C^{0}([0, \xi^{a,+}_{1}]\times\{m^{(a)}\})}+2\|\arctan g'_{+}\|_{C^{0}([0, \xi^{a,+}_{1}])},
\end{aligned}
\end{eqnarray}
where we have used the boundary condition $\eqref{eq:4.26}_3$ on $\widetilde{\Gamma}_{+}$.

\par Next, we are going to estimate $\|D\delta z^{a,(n)}_{+}\|_{C^{0}(\widetilde{\Omega}^{(a)}\cap\Omega_{II})}$.
To do this, we need to take the derivatives on \eqref{eq:4.37} with respect to $(\xi, \eta)$
and then integrate it along the characteristic $\xi=\psi^{i,(n)}_{-}(\eta,\xi^{a})$ from  $m^{(a)}$ to $\eta$ to obtain
\begin{equation}\label{eq:4.39}
\partial_{\xi} \delta z^{a,(n)}_{+}(\psi^{i,(n)}_{-}(\eta,\xi^{a}),\eta)
=\partial_{\xi}\delta z^{a,(n)}_{+}(\xi^{a}, m^{(a)})
-{\int^{\eta}_{m^{(a)}}\partial_{\xi}\big(\frac{1}{\tilde{\lambda}^{a,(n-1)}_{-}}\big)
\partial_{\tau}\delta z^{a,(n)}_{+}(\tau;\xi^a)d\tau ,}
\end{equation}
and
\begin{equation}\label{eq:4.40}
\partial_{\eta} \delta z^{a,(n)}_{+}(\psi^{i,(n)}_{-}(\eta,\xi^{a}),\eta)
=\partial_{\eta}\delta z^{a,(n)}_{+}(\xi^{a}, m^{(a)})
-\int^{\eta}_{m^{(a)}}\partial_{\tau}\big(\frac{1}{\tilde{\lambda}^{a,(n-1)}_{-}}\big)
\partial_{\tau}\delta z^{a,(n)}_{+}(\tau;\xi^a)d\tau,
\end{equation}
where $\partial_{\tau}\delta z^{a,(n)}_{+}(\tau;\xi^a)=\partial_{\tau}\delta z^{a,(n)}_{+}(\psi^{i,(n)}_{-}(\tau,\xi^{a}),\tau)$.
Then by the argument in \emph{Step 2} of the proof of Proposition \ref{prop:4.1} with Lemma \ref{lem:4.2}, we have in $\widetilde{\Omega}_a\cap\Omega_{II}$,
\begin{eqnarray}\label{eq:4.42}
\begin{aligned}
|D\delta z^{a,(n)}_{+} (\psi^{i,(n)}_{-}(\eta,\xi^{a}),\eta)|
&\leq  \mathcal{C}\| D\delta z^{a,(n)}_{+}\|_{C^{0}([0,\xi_1^{a,+}]\times\{m^{(a)}\})},
\end{aligned}
\end{eqnarray}
where the constant $\mathcal{C}$ is positive and only depends on $\widetilde{\underline{U}}$
and $L, \gamma$ provided that $\sigma$ is small depending only $\widetilde{\underline{U}}$ and $L, \gamma$.

Notice that
\begin{eqnarray*}
\begin{aligned}
\partial_{\xi}\delta z^{a,(n)}_{+}(\xi^{a}, m^{(a)})&=\frac{\partial \xi^{a}}{\partial \xi}
\partial_{\xi^{a}}\delta z^{a,(n)}_{+}(\xi^{a}, m^{(a)})\\
&=\frac{\partial \xi^{a}}{\partial \xi}
\Big(\frac{g''_{+}}{1+(g'_{+})^{2}}-\partial_{\xi^{a}}\delta z^{a,(n)}_{-}(\xi^{a}, m^{(a)})\Big),\\
\partial_{\eta}\delta z^{a,(n)}_{+}(\xi^{a}, m^{(a)})&=\frac{\partial \xi^{a}}{\partial \eta}
\partial_{\xi^{a}}\delta z^{a,(n)}_{+}(\xi^{a}, m^{(a)})\\
&=\frac{\partial \xi^{a}}{\partial \eta}\Big(\frac{g''_{+}}{1+(g'_{+})^{2}}
-\partial_{\xi^{a}}\delta z^{a,(n)}_{-}(\xi^{a}, m^{(a)})\Big).
\end{aligned}
\end{eqnarray*}
Then by Lemma \ref{lem:4.2} and the boundary condition $\eqref{eq:4.26}_3$ on $\widetilde{\Gamma}_{+}$, we have
\begin{eqnarray}\label{eq:4.41}
\begin{aligned}
&\| D\delta z^{a,(n)}_{+}(\cdot, m^{(a)})\|_{C^{0}([0,\xi_1^{a,+}])}\\
&\ \ \ \ \leq\|D\xi^{a}\|_{C^{0}([0,\xi_1^{a,+}])}
\| \partial_{\xi^{a}}\delta z^{a,(n)}_{+}\|_{C^{0}([0,\xi_1^{a,+}])}\\
&\ \ \ \ \leq \mathcal{C}\Big(\|\frac{g''_{+}}{1+(g'_{+})^{2}}\|_{C^{0}([0,\xi^{a,+}_{1}])}
+\| \partial_{\xi^{a}}\delta z^{a,(n)}_{-}(\cdot, m^{(a)})\|_{C^{0}([0,\xi^{a,+}_{1}])}\Big)\\
&\ \ \ \ \leq \mathcal{C}\Big(\|g''_{+}\|_{C^{0}([0,\xi^{a,+}_{1}])}
+\| \partial_{\xi^{a}}\delta z^{a,(n)}_{-}(\cdot, m^{(a)})\|_{C^{0}([0,\xi^{a,+}_{1}])}\Big).
\end{aligned}
\end{eqnarray}
%
%
Combining estimates \eqref{eq:4.42} and \eqref{eq:4.41} together yields
%
\begin{eqnarray}\label{eq:4.43}
\begin{aligned}
\|D\delta z^{a,(n)}_{+}\|_{C^{0}(\widetilde{\Omega}_a\cap\Omega_{II})}
\leq \mathcal{C}\big(\|D\delta z^{a,(n)}_{-}\|_{C^{0}([0,\xi^{a,+}_{1}])}
+\|g''_{+}\|_{C^{0}([0,\xi^{a,+}_{1}])}\big),
\end{aligned}
\end{eqnarray}
where the constant $\mathcal{C}$ depends only on $\widetilde{\underline{U}}$, $L$ and $\gamma$.

 Finally, we turn to the estimate of $\|D^{2}\delta z^{a,(n)}_{+}\|_{C^{0}(\widetilde{\Omega}_a\cap\Omega_{II})}$.
We first take the derivatives with respect to
$(\xi, \eta)$ on \eqref{eq:4.39} and \eqref{eq:4.40} twice and then integrate them along the characteristic
$\xi=\psi^{a,(n)}_{-}(\eta,\xi^{a})$ from $m^{(a)}$ to $\eta$ to get
\begin{eqnarray}\label{eq:4.44}
\begin{aligned}
\partial^2_{\xi\xi}\delta z^{a,(n)}_{+}&=\partial^2_{\xi\xi}\delta z^{a,(n)}_{+}(\xi^{a},m^{(a)})
-2\int^{\eta}_{m^{(a)}}\partial_{\xi}\big(\frac{1}{\tilde{\lambda}^{a,(n-1)}_{-}}\big)
\partial^{2}_{\xi\xi} \delta z^{a,(n)}_{+}d\tau\\
&\ \ \ \ -\int^{\eta}_{m^{(a)}}\partial^{2}_{\xi\xi} \big(\frac{1}{\tilde{\lambda}^{a,(n-1)}_{-}}\big)
\partial_{\xi} \delta z^{a,(n)}_{+}d\tau,
\end{aligned}
\end{eqnarray}
\begin{eqnarray}\label{eq:4.45}
\begin{aligned}
\partial^2_{\xi\eta}\delta z^{a,(n)}_{+}&=\partial^2_{\xi\eta}\delta z^{a,(n)}_{+}(\xi^{a},m^{(a)})
-\int^{\eta}_{m^{(a)}}\Big(\partial_{\tau}\big(\frac{1}{\tilde{\lambda}^{a,(n-1)}_{-}}\big)
+\partial_{\xi}\big(\frac{1}{\tilde{\lambda}^{a,(n-1)}_{-}}\big)\Big)\partial^{2}_{\xi\tau} \delta z^{a,(n)}_{+}d\tau\\
&\ \ \ \ -\int^{\eta}_{m^{(a)}}\partial^{2}_{\xi\tau} \big(\frac{1}{\tilde{\lambda}^{a,(n-1)}_{-}}\big)
\partial_{\xi} \delta z^{a,(n)}_{+}d\tau,
\end{aligned}
\end{eqnarray}
and
\begin{eqnarray}\label{eq:4.46}
\begin{aligned}
\partial^2_{\eta\eta}\delta z^{a,(n)}_{+}&=\partial^2_{\eta\eta}\delta z^{a,(n)}_{+}(\xi^{a},m^{(a)})
-2\int^{\eta}_{m^{(a)}}\partial_{\tau}\big(\frac{1}{\tilde{\lambda}^{a,(n-1)}_{-}}\big)
\partial^{2}_{\xi\tau} \delta z^{a,(n)}_{+}d\tau\\
&\ \ \ \ -\int^{\eta}_{m^{(a)}}\partial^{2}_{\tau\tau} \frac{1}{\tilde{\lambda}^{a,(n-1)}_{-}}
\partial_{\xi} \delta z^{a,(n)}_{+}d\tau.
\end{aligned}
\end{eqnarray}
Then by the argument in \emph{Step 3} of the proof of Proposition \ref{prop:4.1} with Lemma \ref{lem:4.2}, we have   in $\widetilde{\Omega}_a\cap\Omega_{II}$,
\begin{eqnarray}\label{eq:4.48}
\begin{aligned}
|D^2\delta z^{a,(n)}_{+} (\psi^{i,(n)}_{-}(\eta,\xi^{a}),\eta)|
&\leq  \mathcal{C}\| D^2\delta z^{a,(n)}_{+}\|_{C^{0}([0,\xi_1^{a,+}]\times\{m^{(a)}\})},
\end{aligned}
\end{eqnarray}
where the positive constant $\mathcal{C}$   only depends on $\widetilde{\underline{U}}$, $L$ and $\gamma$
 provided that $\sigma$ is small depending only $\widetilde{\underline{U}}$, $L$ and $\gamma$.
By the boundary condition $\eqref{eq:4.26}_3$,  one has
\begin{eqnarray*}
\begin{aligned}
\partial^{2}_{\xi\xi}\delta z^{a,(n)}_{+}(\xi^{a}, m^{(a)})
&=\frac{\partial^{2} \xi^{a}}{\partial \xi^{2}}\Big(\frac{g''_{+}}{1+(g'_{+})^{2}}
-\partial_{\xi^{a}}\delta z^{a,(n)}_{-}(\cdot, m^{(a)})\Big)\\
&\ \ \ \  + \Big(\frac{\partial \xi^{a}}{\partial \xi}\Big)^{2}
\Big(\frac{(1+(g'_{+})^{2})g''_{+}-2g'_{+}(g''_{+})^{2}}{(1+(g'_{+})^{2})^{2}}
-\partial^{2}_{\xi^{a}\xi^{a}}\delta z^{a,(n)}_{-}(\cdot, m^{(a)})\Big),\\
\partial^{2}_{\xi\eta}\delta z^{a,(n)}_{+}(\xi^{a}, m^{(a)})
&=\frac{\partial^{2} \xi^{a}}{\partial \xi\partial \eta}\Big(\frac{g''_{+}}{1+(g'_{+})^{2}}
-\partial_{\xi^{a}}\delta z^{a,(n)}_{-}(\cdot, m^{(a)})\Big)\\
&\ \ \ \  + \frac{\partial \xi^{a}}{\partial \xi}\frac{\partial \xi^{a}}{\partial \eta}
\Big(\frac{(1+(g'_{+})^{2})g''_{+}-2g'_{+}(g''_{+})^{2}}{(1+(g'_{+})^{2})^{2}}
-\partial^{2}_{\xi^{a}\xi^{a}}\delta z^{a,(n)}_{-}(\cdot, m^{(a)})\Big),\\
\end{aligned}
\end{eqnarray*}
and
\begin{eqnarray*}
\begin{aligned}
\partial^{2}_{\eta\eta}\delta z^{a,(n)}_{+}(\xi^{a}, m^{(a)})
&=\frac{\partial^{2} \xi^{a}}{\partial^{2} \eta}\Big(\frac{g''_{+}}{1+(g'_{+})^{2}}
-\partial_{\xi^{a}}\delta z^{a,(n)}_{-}(\cdot, m^{(a)})\Big)\\
&\ \ \ \  +\Big(\frac{\partial \xi^{a}}{\partial \eta}\Big)^{2}
\Big(\frac{(1+(g'_{+})^{2})g'''_{+}-2g'_{+}(g''_{+})^{2}}{(1+(g'_{+})^{2})^{2}}
-\partial^{2}_{\xi^{a}\xi^{a}}\delta z^{a,(n)}_{-}(\cdot, m^{(a)})\Big).
\end{aligned}
\end{eqnarray*}
From Lemma \ref{lem:4.2}, we have
\begin{eqnarray}\label{eq:4.47}
\begin{aligned}
&\|D^{2}\delta z^{a,(n)}_{+}(\cdot, m^{(a)})\|_{C^{0}([0,\xi^{a,+}_{1}])}\\
&\ \ \ \ \leq \|D^{2}\xi^{a}\|_{C^{0}(\widetilde{\Omega}^{(a)})}(\|g''_{+}\|_{C^{0}([0,\xi^{a,+}_{1}])}
+\|D\delta z^{a,(n)}_{-}(\cdot, m^{(a)})\|_{C^{0}([0,\xi^{a,+}_{1}])})\\
&\ \ \ \ \ \ \ +\|(D\xi^{a})^{2}\|_{C^{0}(\widetilde{\Omega}^{(a)})}\big(\|g'''_{+}\|_{C^{0}([0,\xi^{a,+}_{1}])}
+\|D^{2}\delta z^{a,(n)}_{-}(\cdot, m^{(a)})\|_{C^{0}([0,\xi^{a,+}_{1}])}\big)\\
&\ \ \ \ \leq \mathcal{C}\Big(\|g''_{+}\|_{C^{1}([0,\xi^{a,+}_{1}])}
+\|D\delta z^{a,(n)}_{-}(\cdot, m^{(a)})\|_{C^{1}([0,\xi^{a,+}_{1}])}\Big),
\end{aligned}
\end{eqnarray}
where the constant $\mathcal{C}$ depends only on $\widetilde{\underline{U}}$, $L$ and $\gamma$.
The estimates \eqref{eq:4.48} and \eqref{eq:4.47} lead to the following estimate:
\begin{eqnarray}\label{eq:4.49b}
\begin{aligned}
\|D^{2}\delta z^{a,(n)}_{+}\|_{C^{0}(\widetilde{\Omega}^{(a)}\cap\Omega_{II})}
 \leq \mathcal{C}\big(\|D\delta z^{a,(n)}_{-}\|_{C^{1}([0,\xi^{a,+}_{1}])}
+\|g''_{+}\|_{C^{1}([0,\xi^{a,+}_{1}])}\big).
\end{aligned}
\end{eqnarray}
Therefore, by \eqref{eq:4.38}, \eqref{eq:4.43} and \eqref{eq:4.49}, there
exist positive constants $\widetilde{C}^{a}_{2,-}$ and $\sigma_{1}$ 
depending only on $\widetilde{\underline{U}}$ and $L$ such that for $\sigma\in (0,\sigma_{1})$,
the estimate \eqref{eq:4.35} holds. This completes the proof of the proposition.
\end{proof}

\begin{remark}\label{remark:4.2}
Following the same argument in the proof above,  
taking the integration along the characteristics up to the contact discontinuity $\{\eta=0\}\cap\Omega_{III}$ for the  equations \eqref{eq:4.37}, \eqref{eq:4.39}--\eqref{eq:4.40} and \eqref{eq:4.44}--\eqref{eq:4.46}, we can actually obtain
\begin{eqnarray*}
\begin{aligned}
&\|\delta z^{a, (n)}_{+}\|_{C^{2}(\Omega_{III}\cap\{\eta=0\})}
+\|\delta z^{b, (n)}_{-}\|_{C^{2}(\Omega_{III}\cap\{\eta=0\})}\\
\leq& \widetilde{C}^{a}_{2,-}\Big(\|\delta z^{a, (n)}_{-}\|_{C^{2}([0,\xi_{1}^{a, +}]\times\{m^{(a)}\})}
+\|g_{+}-1\|_{C^{3}([0, \xi^{a,+}_{1}])}\Big) \\
 &+\widetilde{C}^{a}_{2,+}\Big(\|\delta z^{a, (n)}_{+}\|_{C^{2}([0,\xi_{1}^{b, -}]\times\{-m^{(b)}\})}
+\|g_{-}+1\|_{C^{3}([0, \xi^{b,-}_{1}])}\Big).
\end{aligned}
\end{eqnarray*}
\end{remark}

\subsection{Estimates for the boundary value problem $(\mathbf{\widetilde{P}_{n}})$
in $\Omega_{III}$}
Based on Propositions \ref{prop:4.1} and   \ref{prop:4.2}, in this subsection we will analyze  the boundary value problem $(\mathbf{\widetilde{P}_{n}})$
in $\Omega_{III}$ which involves the contact discontinuity on $\eta=0$. 
The boundary value problem $(\mathbf{\widetilde{P}_{n}})$ in $\Omega_{III}$ can be described
as the following:
\begin{eqnarray}\label{eq:4.49}
(\mathbf{\widetilde{P}_{n}})_{3}
\left\{
\begin{array}{llll}
\partial_{\xi}\delta z^{a,(n)}+diag( \tilde{\lambda}^{a,(n-1)}_{+}, \tilde{\lambda}^{a,(n-1)}_{-})
\partial_{\eta}\delta z^{a,(n)}=0, &  in\ \widetilde{\Omega}^{(a)},  \\
\partial_{\xi}\delta z^{b,(n)}+diag( \tilde{\lambda}^{b,(n-1)}_{+}, \tilde{\lambda}^{b,(n-1)}_{-})
\partial_{\eta}\delta z^{b,(n)}=0, &  in\ \widetilde{\Omega}^{(b)},  \\
\delta z^{a, (n)}_{-}-\delta z^{b, (n)}_{+}=\delta z^{b, (n)}_{-}-\delta z^{a, (n)}_{+},
& on\ \widetilde{\Gamma}_{cd},\\
\alpha^{(n-1)}\delta z^{a,(n)}_{-}+\beta^{(n-1)}\delta z^{b,(n)}_{+}
=\alpha^{(n-1)}\delta z^{a,(n)}_{+}+\beta^{(n-1)}\delta z^{b,(n)}_{-}+c(\xi),
&  on\ \widetilde{\Gamma}_{cd},
\end{array}
\right.
\end{eqnarray}
where $\alpha^{(n-1)}, \beta^{(n-1)}$ and $c(\xi)$ are given  in \eqref{eq:3.48} and \eqref{eq:3.49}.

The way of studying the problem $(\mathbf{\widetilde{P}_{n}})_{3}$ in $\Omega_{III}$ is similar to the way in \S 4.2
except the complicated boundary conditions $\eqref{eq:4.49}_{3}$ and $\eqref{eq:4.49}_4$ on $\widetilde{\Gamma}_{cd}$.
Moreover, following the argument in \S4.1 and \S4.2, we can obtain the estimate of $\delta z^{a,(n)}_+$ in
${\widetilde{\Omega}_a\cap\Omega_{III}}$ and the estimate of $\delta z^{b,(n)}_-$
in ${\widetilde{\Omega}_b\cap\Omega_{III}}$. Then we only need to prove the following proposition.

\begin{proposition}\label{prop:4.3}
There exist constants $\widetilde{C}^{a}_{2,+}>0,\ \widetilde{C}^{b}_{2,-}>0$ and $\sigma_{2}>0$
depending only on $\widetilde{\underline{U}}$, $L$ and $\gamma$ such that for $\sigma\in (0,\sigma_{2})$,
if $\delta z^{i,(n-1)}\in \mathcal{M}_{2\sigma}$ and  $(\delta z^{a,(n)}_{-}, \delta z^{b,(n)}_{+})$ is the solution to
the free boundary value problem $(\mathbf{\widetilde{P}_{n}})_{3}$, then
\begin{eqnarray}\label{eq:4.59}
\begin{aligned}
\|\delta z^{a,(n)}_{-}\|_{C^{2}(\widetilde{\Omega}^{(a)}\cap\Omega_{III})}
& \leq\widetilde{C}^{a}_{2,+}\Big(\|\delta z^{a, (n)}_{+}\|_{C^{2}(\Omega_{III}\cap\{\eta=0\})}
+\|\delta z^{b, (n)}_{-}\|_{C^{2}(\Omega_{III}\cap\{\eta=0\})}\\
&\ \  +\sum_{i=a,b}\big(\|\widetilde{S}^{i}_{0}-\underline{S}^{i}\|_{C^{2}( \Sigma_{i})}
+\|\widetilde{B}^{i}_{0}-\underline{B}^{i}\|_{C^{2}( \Sigma_{i})}\big)\Big),
\end{aligned}
\end{eqnarray}
and
\begin{eqnarray}\label{eq:4.60}
\begin{aligned}
\|\delta z^{b,(n)}_{+}\|_{C^{2}(\widetilde{\Omega}^{(b)}\cap\Omega_{III})}
&\leq\widetilde{C}^{b}_{2,-}\Big(\|\delta z^{a, (n)}_{+}\|_{C^{2}(\Omega_{III}\cap\{\eta=0\})}
+\|\delta z^{b, (n)}_{-}\|_{C^{2}(\Omega_{III}\cap\{\eta=0\})}\\
&\ \  +\sum_{i=a,b}\big(\|\widetilde{S}^{i}_{0}-\underline{S}^{i}\|_{C^{2}( \Sigma_{i})}
+\|\widetilde{B}^{i}_{0}-\underline{B}^{i}\|_{C^{2}( \Sigma_{i})}\big)\Big),
\end{aligned}
\end{eqnarray}
where $\Sigma_{a}=(0, m^{(a)})$ and $\Sigma_{b}=(-m^{(b)},0)$.
\end{proposition}

\par Before giving the proof of Proposition \ref{prop:4.3},   let us introduce the characteristics issuing from $\eta=0$, that is,  $\xi=\psi^{a,(n)}_{+}(\eta,\xi^{a}_{0})$  and $\xi=\psi^{b,(n)}_{-}(\eta,\xi^{b}_{0})$
are the characteristic curves   passing through   the points $(0,\xi_{0})$ defined by
\begin{eqnarray}\label{eq:4.51}
\begin{aligned}
\left\{
\begin{array}{llll}
\frac{d\psi^{a,(n)}_{+}}{d\eta}=\frac{1}{\tilde{\lambda}^{a,(n-1)}_{+}(\eta,\psi^{a,(n)}_{+})},\\
\psi^{a,(n)}_{+}(0,\xi^{a}_{0})=\xi^{a}_{0},
\end{array}
\right.
\end{aligned}
\end{eqnarray}
and
\begin{eqnarray}\label{eq:4.52}
\begin{aligned}
\left\{
\begin{array}{llll}
\frac{d\psi^{b,(n)}_{-}}{d\eta}=\frac{1}{\tilde{\lambda}^{b,(n-1)}_{-}(\eta,\psi^{b,(n)}_{-})},\\
\psi^{b,(n)}_{-}(0,\xi^{b}_{0})=\xi^{b}_{0},
\end{array}
\right.
\end{aligned}
\end{eqnarray}
where $\xi^{a}_{0}, \xi^{b}_{0}\in [0, L]$.
Similarly, we can regard $\xi_0^a$ and $\xi_0^b$ as the functions of $(\xi,\eta)$.
Then similarly to Lemma \ref{lem:4.2}, we also have following estimates for $\xi^{a}_{0}$ and $\xi^{b}_{0}$.
\begin{lemma}\label{lem:4.3}
For any $\delta z^{i,(n-1)}\in \mathcal{M}_{2\sigma}$ $(i=a, b)$, there exist constants $C^{i}_{2,1}$ and $C^{i}_{2,2}$
depending only on $\widetilde{\underline{U}}$, $L$ and $\sigma$ such that
\begin{eqnarray}\label{eq:4.54}
\begin{aligned}
\|D^{k}\xi^{i}_{0}\|_{C^{0}(\widetilde{\Omega}^{(a)}\cup\widetilde{\Omega}^{(b)})}\leq C^{i}_{2,k},\ \ \   k=1,2,
\end{aligned}
\end{eqnarray}
for $i=a, b$.
\end{lemma}

We omit the proof of this lemma since it is similar to that in Lemma \ref{lem:4.2}.. 
Now we   prove Proposition \ref{prop:4.3}.
\begin{proof}[Proof of Proposition \ref{prop:4.3}]
As before we only consider the estimate of $\delta z^{a,(n)}_{-}$ in $\widetilde{\Omega}^{(a)}\cap\Omega_{III}$, since the estimate
of $\delta z^{b,(n)}_{+}$ in $\widetilde{\Omega}^{(b)}\cap\Omega_{III}$ is similar.


\par First of all, from the boundary conditions $\eqref{eq:4.49}_{3}$ and $\eqref{eq:4.49}_4$ on $\widetilde{\Gamma}_{cd}$, we have
\begin{eqnarray}\label{eq:4.61}
\begin{aligned}
&\delta z^{a, (n)}_{-}=\gamma^{(n-1)}_{1}\delta z^{a, (n)}_{+}
+\gamma^{(n-1)}_{3}\delta z^{b, (n)}_{-}+\gamma^{(n-1)}_{4}c,\\
&\delta z^{b, (n)}_{+}=\gamma^{(n-1)}_{2}\delta z^{a, (n)}_{+}
-\gamma^{(n-1)}_{1}\delta z^{b, (n)}_{-}+\gamma^{(n-1)}_{4}c,
\end{aligned}
\end{eqnarray}
where
\begin{eqnarray}\label{eq:4.62}
\begin{aligned}
&\gamma^{(n-1)}_{1}=\frac{\alpha^{(n-1)}-\beta^{(n-1)}}{\alpha^{(n-1)}+\beta^{(n-1)}}, \ \ \
\gamma^{(n-1)}_{2}=\frac{2\alpha^{(n-1)}}{\alpha^{(n-1)}+\beta^{(n-1)}},\\
&\gamma^{(n-1)}_{3}=\frac{2\beta^{(n-1)}}{\alpha^{(n-1)}+\beta^{(n-1)}},\  \  \
\gamma^{(n-1)}_{4}=\frac{1}{\alpha^{(n-1)}+\beta^{(n-1)}}.
\end{aligned}
\end{eqnarray}
As in the proof of Proposition \ref{prop:4.2}, we also rewrite the equation for $\delta z^{a,(n)}_{-}$ as
\begin{eqnarray}\label{eq:4.63}
\begin{aligned}
\partial_{\eta}\delta z^{a,(n)}_{-}+\frac{1}{\tilde{\lambda}^{a,(n-1)}_{+}}\partial_{\xi}\delta z^{a,(n)}_{-}=0.
\end{aligned}
\end{eqnarray}
Then along the characteristic $\xi=\psi^{a,(n)}_{+}(\eta,\xi^{a}_{0})$, 
we have
\begin{eqnarray}\label{eq:4.64}
\begin{aligned}
\|\delta z^{a,(n)}_{-}\|_{C^{0}(\widetilde{\Omega}^{(a)}\cap\Omega_{III})}&=
\|\delta z^{a,(n)}_{-}\|_{C^{0}(\Omega_{III}\cap\{\eta=0\})}\\
&= \|\gamma^{(n-1)}_{1}\delta z^{a, (n)}_{+}
+\gamma^{(n-1)}_{3}\delta z^{b, (n)}_{-}+\gamma^{(n-1)}_{4}c\|_{C^{0}(\widetilde{\Omega}^{(a)}\cap\Omega_{III})}\\
&\leq \mathcal{C}\Big(\|\delta z^{a, (n)}_{+}\|_{C^{0}(\widetilde{\Omega}^{(a)}\cap\Omega_{III})}
+\|\delta z^{b, (n)}_{-}\|_{C^{0}(\widetilde{\Omega}^{(a)}\cap\Omega_{III})}\\
&\ \ \ +\sum_{i=a,b}\big(\|\widetilde{S}^{i}_{0}-\underline{S}^{i}\|_{C^{0}( \Sigma_{i})}
+\|\widetilde{B}^{i}_{0}-\underline{B}^{i}\|_{C^{0}( \Sigma_{i})}\big)\Big),
\end{aligned}
\end{eqnarray}
where we have used the boundary condition \eqref{eq:4.61} on $\widetilde{\Gamma}_{cd}$
and the constant $\mathcal{C}$ depends only on $\widetilde{\underline{U}}$, $L$ and $\gamma$.

 Next, let us consider the estimate of $\|D\delta z^{a,(n)}_{-}\|_{C^{0}(\widetilde{\Omega}^{(a)}\cap\Omega_{III})}$.
Taking the derivatives on equation \eqref{eq:4.63} with respect to $(\xi,\eta)$ and then integrating the resulting equations along the characteristic $\xi=\psi^{a,(n)}_{+}(\eta,\xi^{a}_{0})$ determined by \eqref{eq:4.51},
we obtain
\begin{eqnarray*}
\begin{aligned}
&\partial_{\xi} \delta z^{a,(n)}_{-}=\partial_{\xi}\delta z^{a,(n)}_{-}({\xi^{a}_{0}},0)
-\int^{\eta}_{0}\partial_{\xi}\big(\frac{1}{\tilde{\lambda}^{a,(n-1)}_{+}}\big)
\partial_{\tau}\delta z^{a,(n)}_{-}(\tau,\cdot)d\tau,\\
&\partial_{\eta} \delta z^{a,(n)}_{-}=\partial_{\eta}\delta z^{a,(n)}_{-}({\xi^{a}_{0}}, 0)
-\int^{\eta}_{0}\partial_{\tau}\big(\frac{1}{\tilde{\lambda}^{a,(n-1)}_{+}}\big)
\partial_{\tau}\delta z^{a,(n)}_{-}(\tau,\cdot)d\tau,
\end{aligned}
\end{eqnarray*}
where
\begin{eqnarray*}
\begin{aligned}
\partial_{\xi}\delta z^{a,(n)}_{-}({\xi^{a}_{0}}, 0)&=\frac{\partial {\xi^{a}_{0}}}{\partial \xi}
\partial_{{\xi^{a}_{0}}}\delta z^{a,(n)}_{-}({\xi^{a}_{0}},0),
\quad
\partial_{\eta}\delta z^{a,(n)}_{-}({\xi^{a}_{0}}, 0)&=\frac{\partial{\xi^{a}_{0}}}{\partial \eta}
\partial_{{\xi^{a}_{0}}}\delta z^{a,(n)}_{-}({\xi^{a}_{0}},0).
\end{aligned}
\end{eqnarray*}
Notice that on the contact discontinuity $\eta=0$, by \eqref{eq:4.61},
\begin{eqnarray*}
\begin{aligned}
\partial_{\xi^{a}_{0}}\delta z^{a,(n)}_{-}({\xi^{a}_{0}}, 0)&=\partial_{\xi^{a}_{0}}\gamma^{(n-1)}_{1}\delta z^{a,(n)}_{+}
+\gamma^{(n-1)}_{1}\partial_{\xi^{a}_{0}}\delta z^{a,(n)}_{+}
+\partial_{\xi^{a}_{0}}\gamma^{(n-1)}_{3}\delta z^{b,(n)}_{-}\\
&\ \ \ \ +\gamma^{(n-1)}_{3}\partial_{\xi^{a}_{0}}\delta z^{b,(n)}_{-}+\partial_{\xi^{a}_{0}}\gamma^{(n-1)}_{4}c
+\gamma^{(n-1)}_{4}\partial_{\xi^{a}_{0}}c.
\end{aligned}
\end{eqnarray*}
Then
\begin{eqnarray*}
\begin{aligned}
|D\delta z^{a,(n)}_{-}({\xi^{a}_{0}},0)|
&\leq 
\|D{\xi^{a}_{0}}\|_{C^{0}(\widetilde{\Omega}_a\cap\Omega_{III})}
| \partial_{{\xi^{a}_{0}}}\delta z^{a,(n)}_{-}({\xi^{a}_{0}}, 0)|\\
& \leq \mathcal{C}\Big(\|\delta z^{a, (n)}_{+}\|_{C^{1}((\Omega_{III}\cap\{\eta=0\})}
+\|\delta z^{b, (n)}_{-}\|_{C^{1}((\Omega_{III}\cap\{\eta=0\})}\\
&\ \ \ +\sum_{i=a,b}\big(\|\widetilde{S}^{i}_{0}-\underline{S}^{i}\|_{C^{1}( \Sigma_{i})}
+\|\widetilde{B}^{i}_{0}-\underline{B}^{i}\|_{C^{1}( \Sigma_{i})}\big)\Big),
\end{aligned}
\end{eqnarray*}
where constant $\mathcal{C}$ only depends on $\widetilde{\underline{U}}$, $L$ and $\gamma$.
Following \emph{Step 2} in the proof of Proposition \ref{prop:4.1} and applying the Gronwall type inequality, we can choose $\sigma_1'$ such that for any $\sigma\in(0,\sigma_1')$, we have
\begin{eqnarray}\label{eq:4.68}
\begin{aligned}
\|D\delta z^{a,(n)}_{-}\|_{C^{0}(\widetilde{\Omega}^{(a)}\cap\Omega_{III})} &\leq \mathcal{C}\|D\delta z^{a,(n)}_{-}\|_{C^{0}(\Omega_{III})\cap\{\eta=0\}} \\
& \leq \mathcal{C}\Big(\|\delta z^{a, (n)}_{+}\|_{C^{1}((\Omega_{III}\cap\{\eta=0\})}
+\|\delta z^{b, (n)}_{-}\|_{C^{1}((\Omega_{III}\cap\{\eta=0\})}\\
&\ \ \ +\sum_{i=a,b}\big(\|\widetilde{S}^{i}_{0}-\underline{S}^{i}\|_{C^{1}( \Sigma_{i})}
+\|\widetilde{B}^{i}_{0}-\underline{B}^{i}\|_{C^{1}( \Sigma_{i})}\big)\Big),
\end{aligned}
\end{eqnarray}
where constant $\mathcal{C}$ only depends on $\widetilde{\underline{U}}$, $L$ and $\gamma$.

 Finally, let us consider the estimate of $\|D^{2}\delta z^{a,(n)}_{-}\|_{C^{0}(\widetilde{\Omega}^{(a)}\cap\Omega_{III})}$.
Taking the derivatives on \eqref{eq:4.61} with respect to
$(\xi, \eta)$ twice and then integrating the resulting equations along the characteristic
$\xi=\psi^{a,(n)}_{+}(\eta,{\xi^{a}_{0}})$ from $0$ to $\eta$, we obtain
\begin{eqnarray*}
\begin{aligned}
\partial^2_{\xi\xi}\delta z^{a,(n)}_{-}&=\partial^2_{\xi\xi}\delta z^{a,(n)}_{-}(\xi^{a}_{0},0)
-2\int^{\eta}_{0}\partial_{\xi}\big(\frac{1}{\tilde{\lambda}^{a,(n-1)}_{+}}\big)
\partial^{2}_{\xi\xi} \delta z^{a,(n)}_{-}d\tau\\
&\ \ \ \
-\int^{\eta}_{0}\partial^{2}_{\xi\xi} \big(\frac{1}{\tilde{\lambda}^{a,(n-1)}_{+}}\big)
\partial_{\xi} \delta z^{a,(n)}_{-}d\tau,\\
\partial^2_{\xi\eta}\delta z^{a,(n)}_{-}&=\partial^2_{\xi\eta}\delta z^{a,(n)}_{-}(\xi^{a}_{0},0)
-\int^{\eta}_{0}\Big(\partial_{\tau}\big(\frac{1}{\tilde{\lambda}^{a,(n-1)}_{+}}\big)
+\partial_{\xi}\big(\frac{1}{\tilde{\lambda}^{a,(n-1)}_{+}}\big)\Big)\partial^{2}_{\xi\tau} \delta z^{a,(n)}_{-}d\tau\\
&\ \ \ \ -\int^{\eta}_{0}\partial^{2}_{\xi\tau} \big(\frac{1}{\tilde{\lambda}^{a,(n-1)}_{+}}\big)
\partial_{\xi} \delta z^{a,(n)}_{-}d\tau,
\end{aligned}
\end{eqnarray*}
and
\begin{eqnarray*}
\begin{aligned}
\partial^2_{\eta\eta}\delta z^{a,(n)}_{-}&=\partial^2_{\eta\eta}\delta z^{a,(n)}_{+}(\xi^{a}_{0},0)
-2\int^{\eta}_{0}\partial_{\tau}\big(\frac{1}{\tilde{\lambda}^{a,(n-1)}_{+}}\big)
\partial^{2}_{\xi\tau} \delta z^{a,(n)}_{-}d\tau
-\int^{\eta}_{0}\partial^{2}_{\tau\tau} \frac{1}{\tilde{\lambda}^{a,(n-1)}_{+}}
\partial_{\xi} \delta z^{a,(n)}_{-}d\tau.
\end{aligned}
\end{eqnarray*}
For the terms $D^{2}\delta z^{a,(n)}_{-}(\xi^{a}_{0}, 0)$ in above equations, we first have
\begin{eqnarray*}
\begin{aligned}
\partial^{2}_{\xi\xi}\delta z^{a,(n)}_{-}({\xi^{a}_{0}}, 0)
&=\frac{\partial^{2} \xi^{a}_{0}}{\partial \xi^{2}}\partial_{\xi^{a}_{0}}\delta z^{a,(n)}_{-}(\xi^{a}_{0},0)
+ \Big(\frac{\partial \xi^{a}_{0}}{\partial \xi}\Big)^{2}
\partial^{2}_{\xi^{a}_{0}\xi^{a}_{0}}\delta z^{a,(n)}_{-}(\xi^{a}_{0},0),\\
\partial^{2}_{\xi\eta}\delta z^{a,(n)}_{-}(\xi^{a}_{0}, 0)
&=\frac{\partial^{2} \xi^{a}_{0}}{\partial \xi\partial \eta}\partial_{\xi^{a}_{0}}\delta z^{a,(n)}_{-}(\xi^{a}_{0},0)
 + \frac{\partial \xi^{a}_{0}}{\partial \xi}\frac{\partial \xi^{a}_{0}}{\partial \eta}
\partial^{2}_{\xi^{a}_{0}\xi^{a}_{0}}\delta z^{a,(n)}_{-}(\xi^{a}_{0},0),\\
\partial^{2}_{\eta\eta}\delta z^{a,(n)}_{-}(\xi^{a}_{0}, 0)
&=\frac{\partial^{2} \xi^{a}_{0}}{\partial \eta^{2}}\partial_{\xi^{a}_{0}}\delta z^{a,(n)}_{-}(\xi^{a}_{0},0)
+ \Big(\frac{\partial \xi^{a}_{0}}{\partial \eta}\Big)^{2}
\partial^{2}_{\xi^{a}_{0}\xi^{a}_{0}}\delta z^{a,(n)}_{-}(\xi^{a}_{0},0).
\end{aligned}
\end{eqnarray*}
For the term $\partial^{2}_{\xi^{a}_{0}\xi^{a}_{0}}\delta z^{a,(n)}_{-}(\xi^{a}_{0},0)$, taking the  derivatives twice on \eqref{eq:4.61} gives
\begin{eqnarray*}
\begin{aligned}
\partial^{2}_{\xi^{a}_{0}\xi^{a}_{0}}\delta z^{a,(n)}_{-}(\xi^{a}_{0},0)
&=\partial_{\xi^{a}_{0}}\Big(\partial_{\xi^{a}_{0}}\gamma^{(n-1)}_{1}\delta z^{a,(n)}_{+}
+\gamma^{(n-1)}_{1}\partial_{\xi^{a}_{0}}\delta z^{a,(n)}_{+}
+\partial_{\xi^{a}_{0}}\gamma^{(n-1)}_{3}\delta z^{b,(n)}_{-}\\
&\ \ \ \ +\gamma^{(n-1)}_{3}\partial_{\xi^{a}_{0}}\delta z^{b,(n)}_{-}
+\partial_{\xi^{a}_{0}}\gamma^{(n-1)}_{4}c
+\gamma^{(n-1)}_{4}\partial_{\xi^{a}_{0}}c\Big)\\
&=\partial^{2}_{\xi^{a}_{0}\xi^{a}_{0}}\gamma^{(n-1)}_{1}\delta z^{a,(n)}_{+}
+2\partial_{\xi^{a}_{0}}\gamma^{(n-1)}_{1}\partial_{\xi^{a}_{0}}\delta z^{a,(n)}_{+}
+\gamma^{(n-1)}_{1}\partial^{2}_{\xi^{a}_{0}\xi^{a}_{0}}\delta z^{a,(n)}_{+}\\
&\ \ \ \ + \partial^{2}_{\xi^{a}_{0}\xi^{a}_{0}}\gamma^{(n-1)}_{3}\delta z^{b,(n)}_{-}
+2\partial_{\xi^{a}_{0}}\gamma^{(n-1)}_{3}\partial_{\xi^{a}_{0}}\delta z^{b,(n)}_{-}
+\gamma^{(n-1)}_{3}\partial^{2}_{\xi^{a}_{0}\xi^{a}_{0}}\partial_{\xi^{a}_{0}}\delta z^{b,(n)}_{-}\\
&\ \ \ \ + \partial^{2}_{\xi^{a}_{0}\xi^{a}_{0}}\gamma^{(n-1)}_{4}c
+2\partial_{\xi^{a}_{0}}\gamma^{(n-1)}_{4}\partial_{\xi^{a}_{0}}c
+\gamma^{(n-1)}_{4}\partial^{2}_{\xi^{a}_{0}\xi^{a}_{0}}c.
\end{aligned}
\end{eqnarray*}
Thus,
\begin{eqnarray*}
\begin{aligned}
|\partial^{2}_{\xi^{a}_{0}\xi^{a}_{0}}\delta z^{a,(n)}_{-}(\xi_0^a,0)|
&\leq \mathcal{C}\Big(\|\delta z^{a, (n)}_{+}\|_{C^{2}(\Omega_{III}\cap\{\eta=0\})}
+\|\delta z^{b, (n)}_{-}\|_{C^{2}(\Omega_{III}\cap\{\eta=0\})}\\
&\ \ \ +\sum_{i=a,b}\big(\|\widetilde{S}^{i}_{0}-\underline{S}^{i}\|_{C^{2}( \Sigma_{i})}
+\|\widetilde{B}^{i}_{0}-\underline{B}^{i}\|_{C^{2}( \Sigma_{i})}\big)\Big).
\end{aligned}
\end{eqnarray*}
Then,  by Lemma \ref{lem:4.3}, on $\Omega_{III}\cap\{\eta=0\}$, one has
\begin{eqnarray}\label{eq:4.69}
\begin{aligned}
|D^{2}\delta z^{a,(n)}_{-}(\xi_0^{a},0)|
& \leq \|D^{2}\xi^{a}_{0}\|_{C^{0}(\widetilde{\Omega}^{(a)})}
\|\partial_{\xi^{a}_{0}}\delta z^{a,(n)}_{-}(\cdot,0)\|_{C^{0}(\Omega_{III}\cap\{\eta=0\})}\\
&\ \ \ +\|(D\xi^{a}_{0})^{2}\|_{C^{0}(\widetilde{\Omega}^{(a)})}\|\partial^{2}_{\xi^{a}_{0}\xi^{a}_{0}}\delta z^{a,(n)}_{-}(\cdot,0)\|_{C^{0}(\Omega_{III}\cap\{\eta=0\})}\\
&\leq \mathcal{C}\Big(\|\delta z^{a, (n)}_{+}\|_{C^{2}(\Omega_{III}\cap\{\eta=0\})}
+\|\delta z^{b, (n)}_{-}\|_{C^{2}(\Omega_{III}\cap\{\eta=0\})}\\
&\ \ \ +\sum_{i=a,b}\big(\|\widetilde{S}^{i}_{0}-\underline{S}^{i}\|_{C^{2}( \Sigma_{i})}
+\|\widetilde{B}^{i}_{0}-\underline{B}^{i}\|_{C^{2}( \Sigma_{i})}\big)\Big),
\end{aligned}
\end{eqnarray}
where the constant $\mathcal{C}$ depends only on $\widetilde{\underline{U}}$, $L$ and $\gamma$.

\par With   \eqref{eq:4.69} and following \emph{Step 3} in the proof of Proposition \ref{prop:4.1}
 and further applying the Gronwall inequality,  
we can choose  the constant $\sigma''_{1}$ sufficiently small such that 
for any $\sigma \in (0, \sigma''_{1})$, we have
\begin{eqnarray}\label{eq:4.70}
\begin{aligned}
\|D^{2}\delta z^{a,(n)}_{+}\|_{C^{0}(\widetilde{\Omega}^{(a)}\cap\Omega_{III})}
&\leq  2\mathcal{C}\Big(\|\delta z^{a, (n)}_{+}\|_{C^{2}(\Omega_{III}\cap\{\eta=0\})}
+\|\delta z^{b, (n)}_{-}\|_{C^{2}(\Omega_{III}\cap\{\eta=0\})}\\
&\ \ \ +\sum_{i=a,b}\big(\|\widetilde{S}^{i}_{0}-\underline{S}^{i}\|_{C^{2}( \Sigma_{i})}
+\|\widetilde{B}^{i}_{0}-\underline{B}^{i}\|_{C^{2}( \Sigma_{i})}\big)\Big).
\end{aligned}
\end{eqnarray}
From the estimates \eqref{eq:4.64}, \eqref{eq:4.68} and \eqref{eq:4.70} together, there
exist constants $\widetilde{C}^{a}_{2,+}>0$ and $\sigma_{2}=\min\{\sigma'_{1},\sigma''_{1},1\}>0$
depending only on $\widetilde{\underline{U}}$, $L$ and $\gamma$, such that for $\sigma\in (0,\sigma_{2})$,
the estimate \eqref{eq:4.59} holds. This completes the proof of the proposition.
\end{proof}

With Propositions \ref{prop:4.1}-\ref{prop:4.3}, we can prove Theorem \ref{thm:4.1}.

\begin{proof}[Proof of Theorem \ref{thm:4.1}]
We will prove inequality \eqref{eq:4.71} by the induction procedure together with Propositions \ref{prop:4.1}-\ref{prop:4.3}.

Note that by Remark \ref{remark:4.2}, $\|\delta z^{a, (n)}_{+}\|_{C^{2}(\Omega_{III}\cap\{\eta=0\})}
+\|\delta z^{b, (n)}_{-}\|_{C^{2}(\Omega_{III}\cap\{\eta=0\})}$ in \eqref{eq:4.59}--\eqref{eq:4.60} can be bounded by the terms on the right hand side of the inequality \eqref{eq:4.35}--\eqref{eq:4.36}, which can finally be bounded from Proposition \ref{prop:4.1}. Thus,
\begin{eqnarray}\label{equ:4.59a}
\begin{aligned}
&\|\delta z^{a,(n)}\|_{C^{2}( \Omega^{(a)}\cap(\Omega_I\cup\Omega_{II}\cup\Omega_{III}))}+\|\delta z^{b,(n)}\|_{C^{2}( \Omega^{(b)}\cap(\Omega_I\cup\Omega_{II}\cup\Omega_{III}))}\\
&\ \ \ \leq C^{*}_{1}\Big(\sum_{i=a,b}\big(\|z^{i}_{0}-\underline{z}^{i}\|_{C^{2}(\Sigma_{i})}
+\|\widetilde{S}^{i}_{0}-\underline{S}^{i}\|_{C^{2}( \Sigma_{i})}+\|\widetilde{B}^{i}_{0}-\underline{B}^{i}\|_{C^{2}( \Sigma_{i})}\big)\\
&\ \ \ \ \ \  
+\|g_{+}-1\|_{C^{3}([0, \xi^{a,+}_{1}])}+\|g_{-}+1\|_{C^{3}([0, \xi^{b,-}_{1}])}\Big).
\end{aligned}
\end{eqnarray}
Let $\xi_1^*;=\min\{\xi_1^{a,+},\,\xi_1^{b,-}\}$ and $\Omega_1^{(i)}:=\widetilde{\Omega}^{(i)}\cap\{\xi\leq\xi^*_1\}$
for $i=a$ or $b$. In $\Omega_{IV}\cap\Omega_1^{(i)}$, we can repeat the proof of Proposition \ref{prop:4.1}
to show that inequality \eqref{equ:4.59a} also holds. Then,  there exist constants $C^{*}_{1}$ and
$\sigma^{*}_{1}$ which only depend on $\widetilde{\underline{U}},\ L$ and $\gamma$, such that for $\sigma\in (0,\sigma^{*}_{1})$,
we have
\begin{eqnarray*}
\begin{aligned}
&\|\delta z^{a,(n)}\|_{C^{2}( \Omega^{(a)}_{1})}+\|\delta z^{b,(n)}\|_{C^{2}( \Omega^{(b)}_{1})}\\
&\ \ \ \leq C^{*}_{1}\Big(\sum_{i=a,b}\big(\|z^{i}_{0}-\underline{z}^{i}\|_{C^{2}(\Sigma_{i})}
+\|\widetilde{S}^{i}_{0}-\underline{S}^{i}\|_{C^{2}( \Sigma_{i})}\big)\\
&\ \ \ \ \ \  +\sum_{i=a,b}\|\widetilde{B}^{i}_{0}-\underline{B}^{i}\|_{C^{2}( \Sigma_{i})}
+\|g_{+}-1\|_{C^{3}([0, \xi^{*}_{1}])}+\|g_{-}+1\|_{C^{3}([0, \xi^{*}_{1}])}\Big).
\end{aligned}
\end{eqnarray*}
Let $\xi_2^{a, +}$ be the intersection point of the characteristic corresponding to $\lambda_+$ starting from
$(\xi_1^*,0)$ and the upper nozzle wall $\eta=m^{(a)}$,  $\xi_2^{b,-}$  be the intersection point of the
characteristic corresponding to $\lambda_-$ starting from $(\xi_1^*,0)$ and the lower nozzle wall
$\eta=-m^{(b)}$,  and  $\Omega_2^{(i)}:=\widetilde{\Omega}^{(i)}\cap \{\xi_1^*\leq\xi\leq \xi^*_2\}$
for $i=a$ or $b$. Regarding the line $\xi=\xi_1^*$ as the initial line $\xi=0$ and repeating the argument of the proof of
 Proposition \ref{prop:4.1}--\ref{prop:4.2} again, one has
\begin{eqnarray}\label{eq:4.72a}
\begin{aligned}
&\|\delta z^{a,(n)}\|_{C^{2}( \Omega^{(a)}_{2})}+\|\delta z^{b,(n)}\|_{C^{2}( \Omega^{(b)}_{2})}\\
&\ \ \ \leq C_{2}\Big(\sum_{i=a,b}\big(\|z^{i}_{0}-\underline{z}^{i}\|_{C^{2}(\Omega^{(i)}\cap\{\xi=\xi_1^*\})}
+\|\widetilde{S}^{i}_{0}-\underline{S}^{i}\|_{C^{2}( \Omega^{(i)}\cap\{\xi=\xi_1^*\})}\big)\\
&\ \ \ \ \ \  +\sum_{i=a,b}\|\widetilde{B}^{i}_{0}-\underline{B}^{i}\|_{C^{2}( \Omega^{(i)}\cap\{\xi=\xi_1^*\})}
+\|g_{+}-1\|_{C^{3}([\xi^{*}_{1},\xi^*_2])}+\|g_{-}+1\|_{C^{3}([\xi^{*}_{1},\xi_2^*])}\Big)\\
&\ \ \ \leq C^{*}_{2}\Big(\sum_{i=a,b}\big(\|z^{i}_{0}-\underline{z}^{i}\|_{C^{2}(\Sigma_{i})}
+\|\widetilde{S}^{i}_{0}-\underline{S}^{i}\|_{C^{2}( \Sigma_{i})}\big)\\
&\ \ \ \ \ \  +\sum_{i=a,b}\|\widetilde{B}^{i}_{0}-\underline{B}^{i}\|_{C^{2}( \Sigma_{i})}
+\|g_{+}-1\|_{C^{3}([0, \xi^{*}_{2}])}+\|g_{-}+1\|_{C^{3}([0, \xi^{*}_{2}])}\Big).
\end{aligned}
\end{eqnarray}
Then, we can repeat this procedure for $\ell$ times with $\ell=[\frac{L}{\xi^{*}_{1}}]+1$.
Define $\xi_\ell^*$ and $\Omega_\ell^{(i)}$ for $i=a$, or $b$.
Obviously $\widetilde{\Omega}^{(i)}=\cup_{1\leq k\leq \ell}\Omega_k^{(i)}$.
Summing all the estimates \eqref{eq:4.72a} together for $k=1$, $...$, $\ell$, we finally obtain \eqref{eq:4.71}.
 This completes the proof of the theorem.
\end{proof}

\section{Convergence of the Approximate Solution and Existence of Problem $(\widetilde{\mathbf{P}})$}

In this section, we shall prove that the mapping $\mathcal{T}$ defined by \eqref{eq:3.51} is a contraction map,
so that the   sequence of solutions obtained in \S3.3 converges to a limit function which is
actually the solution to the problem $(\widetilde{\mathbf{P}})$.

\par First we will show the following proposition.

\begin{proposition}\label{prop:5.1}
Under the assumptions of Theorem \ref{thm:3.2}, there exists a constant $\widetilde{\sigma}_{0}>0$ depending only on
$\underline{\widetilde{U}}$,\ $L$ and $\gamma$ such that for any $\sigma \in(0,\widetilde{\sigma}_{0})$,
the sequence $\big\{(z^{a,(n)}, z^{b,(n)})\big\}^{\infty}_{n=1}$ constructed
by the iteration scheme $(\mathbf{\widetilde{P}_{n}})$ is well defined in $\mathcal{M}_{2\sigma}$.
Moreover, the sequence $\big\{(z^{a,(n)}, z^{b,(n)})\big\}^{\infty}_{n=1}$
is convergent in $C^{1}\big(\widetilde{\Omega}^{(a)}\big)\times C^{1}\big(\widetilde{\Omega}^{(b)}\big)$
and its limit function is the solution of the problem $(\widetilde{\mathbf{P}})$.
\end{proposition}

\begin{proof}
Firstly, by Theorem \ref{thm:4.1}, we have
\begin{eqnarray*}
\begin{aligned}
&\|\delta z^{a,(n)}\|_{C^{2}(\widetilde{\Omega}^{(a)})}+\|\delta z^{b,(n)}\|_{C^{2}(\widetilde{\Omega}^{(b)})}\\
&\ \ \ \leq C^{*}\Big(\sum_{i=a,b}\big(\|z^{i}_{0}-\underline{z}^{i}\|_{C^{2}(\Sigma_{i})}
+\|\widetilde{S}^{i}_{0}-\underline{S}^{i}\|_{C^{2}( \Sigma_{i})}\big)\\
&\ \ \ \ \ \  +\sum_{i=a,b}\|\widetilde{B}^{i}_{0}-\underline{B}^{i}\|_{C^{2}( \Sigma_{i})}
+\|g_{+}-1\|_{C^{3}([0, L])}+\|g_{-}+1\|_{C^{3}([0, L])}\Big)\\
&\ \ \ \leq C^{*}\widetilde{\varepsilon},
\end{aligned}
\end{eqnarray*}
where the constant $C^*$ only depends on $\widetilde{\underline{U}},\ L$ and $\gamma$.
Thus, we can choose $\widetilde{\varepsilon}$ sufficiently small such that $C^{*}\widetilde{\varepsilon}\leq 2\sigma$,
then this implies that the iteration scheme $(\mathbf{\widetilde{P}_{n}})$ is
well defined in $\mathcal{M}_{2\sigma}$, that is, 
the map
$\mathcal{T}:\ \mathcal{M}_{2\sigma} \longmapsto \mathcal{M}_{2\sigma}$
given in \eqref{eq:3.51} is well defined. We set
$(z^{a,(n+1)}, z^{b,(n+1)}):=\mathcal{T}(z^{a,(n)}, z^{b,(n)})$.

Secondly, we will show that $\mathcal{T}$ is a contraction mapping which implies that the sequence
$\big\{(z^{a,(n)}, z^{b,(n)})\big\}^{\infty}_{n=0}$ is convergent in
$C^{1}(\widetilde{\Omega}^{(a)})\times C^{1}(\widetilde{\Omega}^{(b)})$.
To this end, let
\begin{eqnarray*}
\Delta z^{i,(n)}:=\delta z^{i, (n)}-\delta z^{i,(n-1)},\quad   i=a,\ b.
\end{eqnarray*}
Then, $(\Delta z^{a,(n)},\Delta z^{b,(n)})$ satisfies following boundary value problem,
\begin{eqnarray*}\label{eq:7.5}
\begin{aligned}
(\mathbf{\Delta P}_{n})\quad
\left\{
\begin{array}{llll}
\partial_{\xi}\Delta z^{a,(n+1)}+diag\big(\tilde{\lambda}^{a,(n)}_{+},\tilde{\lambda}^{a,(n)}_{-}\big)
\partial_{\eta}\Delta z^{a,(n+1)}\\
\ \ \  \ =diag \big(\tilde{\lambda}^{a,(n-1)}_{+}-\tilde{\lambda}^{a,(n)}_{+},
\tilde{\lambda}^{a,(n-1)}_{-}-\tilde{\lambda}^{a,(n)}_{-} \big)\partial_{\eta}\delta z^{a,(n)},
&\ \ \ in\ \widetilde{\Omega}^{(a)},\\
\partial_{\xi}\Delta z^{b,(n+1)}+diag\big(\tilde{\lambda}^{b,(n)}_{+},\tilde{\lambda}^{b,(n)}_{-}\big)
\partial_{\eta}\Delta z^{b,(n+1)}\\
\ \ \  \ =diag \big(\tilde{\lambda}^{b,(n-1)}_{+}-\tilde{\lambda}^{b,(n)}_{+},
\tilde{\lambda}^{b,(n-1)}_{-}-\tilde{\lambda}^{b,(n)}_{-} \big)\partial_{\eta}\delta z^{b,(n)},
&\ \ \ in\ \widetilde{\Omega}^{(b)},\\
\Delta z^{a,(n+1)}=0,  &\ \ \  on\ \xi=0,\\
\Delta z^{b, (n+1)}=0,   &\ \ \  on\ \xi=0, \\
\Delta z^{a,(n+1)}_{-}+\Delta z^{a,(n+1)}_{+}=0, &\ \ \ on\ \widetilde{\Gamma}_{+},\\
\Delta z^{b,(n+1)}_{-}+\Delta z^{b,(n+1)}_{+}=0, &\ \ \ on\ \widetilde{\Gamma}_{-},\\
\Delta z^{a, (n+1)}_{-}-\Delta z^{b, (n+1)}_{+}=\Delta z^{b, (n+1)}_{-}-\Delta z^{a, (n+1)}_{+},
&\ \ \ on\ \widetilde{\Gamma}_{cd},\\
\alpha^{(n)}\Delta z^{a,(n+1)}_{-}+\beta^{(n)}\Delta z^{b,(n+1)}_{+}\\
\ \  =\alpha^{(n)}\Delta z^{a,(n+1)}_{+}+\beta^{(n)}\Delta z^{b,(n+1)}_{-}\\
\ \ \ \ \ \ +(\alpha^{(n)}-\alpha^{(n-1)})(\delta z^{a,(n)}_{+}-\delta z^{a,(n)}_{-})\\
\ \ \ \ \ \ +(\beta^{(n)}-\beta^{(n-1)})(\delta z^{b,(n)}_{-}-\delta z^{b,(n)}_{+}),
&\ \ \ on\ \widetilde{\Gamma}_{cd}.
\end{array}
\right.
\end{aligned}
\end{eqnarray*}
Similarly to the proof of Theorem \ref{thm:4.1}, there exists a constant $C>0$
depending only on $\widetilde{\underline{U}}$ and $L$, such that
\begin{eqnarray*}
\begin{aligned}
&\big\|\Delta z^{a,(n+1)}\big\|_{C^{1}(\widetilde{\Omega}^{(a)})}
+\big\|\Delta z^{b,(n+1)}\big\|_{C^{1}(\widetilde{\Omega}^{(b)})}\\
&\ \ \ \ \leq C\big(\sum_{i=a,b}\big\|\delta z^{i,(n)}\big\|_{C^{2}(\widetilde{\Omega}^{(i)})}\big)
\big(\big\|\Delta z^{a,(n+1)}\big\|_{C^{1}(\widetilde{\Omega}^{(a)})}
+\big\|\Delta z^{b,(n+1)}\big\|_{C^{1}(\widetilde{\Omega}^{(b)})}\big)\\
&\ \ \ \ \leq 2C\sigma\big(\big\|\Delta z^{a,(n)}\big\|_{C^{1}(\widetilde{\Omega}^{(a)})}
+\big\|\Delta z^{b,(n)}\big\|_{C^{1}(\widetilde{\Omega}^{(b)})}\big).
\end{aligned}
\end{eqnarray*}
Take $\sigma_{0}=\frac{1}{4C}$. Then for any $0<\sigma\leq\sigma_{0}$,
we get
\begin{eqnarray}\label{equ:5.2}
\begin{aligned}
&\big\|\Delta z^{a,(n+1)}\big\|_{C^{1}(\widetilde{\Omega}^{(a)})}
+\big\|\Delta z^{b,(n+1)}\big\|_{C^{1}(\widetilde{\Omega}^{(b)})}
\leq \frac{1}{2}\big(\big\|\Delta z^{a,(n)}\big\|_{C^{1}(\widetilde{\Omega}^{(a)})}
+\big\|\Delta z^{b,(n)}\big\|_{C^{1}(\widetilde{\Omega}^{(b)})}\big).
\end{aligned}
\end{eqnarray}
\emph{i.e.},
\begin{eqnarray*}
\begin{aligned}
\big\|\mathcal{T}(\Delta z^{a,(n)}, \Delta z^{b,(n)})\big\|_{C^{1}(\widetilde{\Omega}^{(a)}\cup \widetilde{\Omega}^{(b)})}
\leq \frac{1}{2}\big\|(\Delta z^{a,(n)}, \Delta z^{b,(n)})\big\|_{C^{1}(\widetilde{\Omega}^{(a)}\cup \widetilde{\Omega}^{(b)})}.
\end{aligned}
\end{eqnarray*}
This implies that $\{(\delta z^{a,(n)}, \delta z^{b,(n)})\}^{\infty}_{n=1}$ is a Cauchy sequence
and has a limit which is denoted by $(\delta z^{a}, \delta z^{b})\in C^{1}(\widetilde{\Omega}^{(a)})
\times C^{1}(\widetilde{\Omega}^{(b)})$.

Finally, define
\begin{eqnarray*}
z^{a}:=\delta z^{a}+\underline{z}^{a},\ z^{b}:=\delta z^{b}+\underline{z}^{b}.
\end{eqnarray*}
Obviously, $(z^{a}, z^{b})\in \widetilde{\Omega}^{(a)}\times \widetilde{\Omega}^{(b)}$
is a weak solution of problem $(\widetilde{\mathbf{P}})$
in $\widetilde{\Omega}$.
This completes the proof of the proposition.
\end{proof}

 Now we are ready to prove Theorem \ref{thm:3.2}.
\begin{proof}[ Proof of Theorem \ref{thm:3.2}]
The existence is a direct consequence of Proposition \ref{prop:5.1}.
For the uniqueness, we consider two solutions $z_{1}=(z^{a}_{1}, z^{b}_{1})$ and
$z_{2}=(z^{a}_{2}, z^{b}_{2})$ which both satisfy the estimate \eqref{eq:3.56}. Define
$$
\Delta z^{(i)}:=\delta z^{(i)}_1-\delta z^{(i)}_2,\quad   i=a,\ b.
$$
Obviously, following the argument in the proof of Proposition \ref{prop:5.1}, we know that $\Delta z^{(i)}$
satisfies estimates \eqref{equ:5.2} for $C^{0}$-estimates, where $\Delta z^{i,(n+1)}$ and $\Delta z^{i,(n)}$
are replaced by
$\Delta z^{(i)}$ for $i=a$ or $b$. It means that $\Delta z^{(i)}=0$, then we conclude that $z_{1}=z_{2}$.
This completes the proof of Theorem \ref{thm:3.2}.
\end{proof}

\bigskip

\appendix

\section{Blow-up of Solution in a Semi-infinity Long Nozzle} 

In this appendix, we  give an example to show that the solution for the semi-infinity long nozzle problem
will blow-up in general. More precisely,  the irrotational flow in a semi-infinitely long flat nozzle (see Fig.\ref{fig6.1}) with compatibility conditions   will blow up if it is not a constant flow.

\begin{figure}[ht]
\begin{center}\label{fig1}
\begin{tikzpicture}[scale=1.2]
\draw [thin][->] (-5,0)to[out=30, in=-150](-4,0.1);
\draw [thin][->] (-5,-0.5)to[out=30, in=-150](-4,-0.4);
\draw [thin][->] (-5.0,-1)to[out=30, in=-150](-4,-0.9);
\draw [thin][->] (-5.0,-1.5)to[out=30, in=-150](-4,-1.4);
\draw [line width=0.06cm] (-3.7,0.6) --(2.4,0.6);
\draw [line width=0.06cm] (-3.7,-1.8) --(2.4,-1.8);
\draw [thin][dashed] (-3.7,-1.8) --(-3.7,0.6);
\node at (-5.0, 0.4) {$U_{0}(y)$};
\node at (2.9, 0.6) {$y=1$};
\node at (2.9, -1.8) {$y=0$};
\node at (-1, -0.4) {$\Omega$};
\node at (-3.7, -2.2) {$x=0$};
\end{tikzpicture}
\end{center}
\caption{Supersonic flow past a semi-infinitely long straight nozzle}\label{fig6.1}
\end{figure}
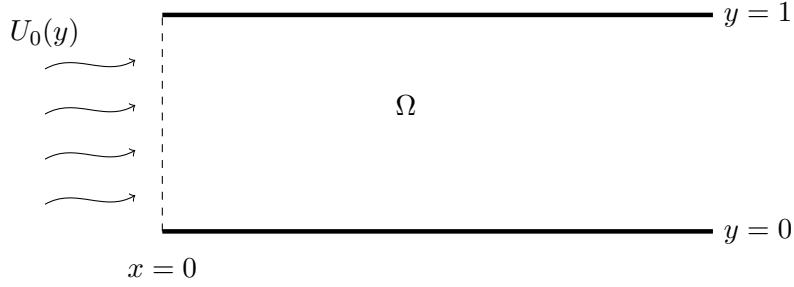

\par Consider the supersonic incoming flow past a semi-infinity long flat nozzle 
 $
\Omega:=\big\{(x,y)\in \mathbb{R}^{2}: 0<y<1,\ x>0 \big\},
$ 
whose lower and upper boundaries are denoted by $\Gamma_{a}$ and $\Gamma_{b}$, with 
$
\Gamma_{a}:=\{(x,y):y=1,\ x>0 \}$ and $\Gamma_{b}:=\{(x,y):y=0,\ x>0 \}.
$ 
Assume that the flow is irrotational satisfying
\begin{eqnarray}\label{eq:6.1}
\left\{
\begin{array}{llll}
     \partial_x(\rho u)+\partial_y(\rho v)=0, \\
    \partial_{x}v-\partial_{y}u=0,
     \end{array}
     \right.
\end{eqnarray}
together with the Bernoulli  law
\begin{eqnarray}\label{eq:6.2}
\frac{u^{2}+v^{2}}{2}+\frac{c^{2}}{\gamma-1}=\frac{1}{2}\hat{q}^{2},
\end{eqnarray}
where $(u,v), \rho$ stand for the velocity and density respectively, and $c$ represents the sonic speed which
is given by $c=\rho^{\frac{\gamma-1}{2}}$ while $\hat{q}=\sqrt{\frac{\gamma+1}{\gamma-1}}\hat{c}$ and $\hat{c}$
is the critical sonic speed.

\par The incoming flow at the inlet $x=0$ is given by
\begin{eqnarray}\label{eq:6.9}
\big(u,v, \rho\big)(0,y):=\big(u_{0},v_{0}, \rho_{0}\big)(y)\in C^{2}([0,1]),
\end{eqnarray}
and along the upper and lower walls of the straight {semi-infinitely} long nozzle the flow speeds satisfy the following boundary condition
\begin{eqnarray}\label{eq:6.10}
v(x,1)=v(x,0)=0,
\end{eqnarray}
for $x\geq0$. Then by \eqref{eq:6.2} and \eqref{eq:6.10}, {we impose the compatibility condition on the first order derivatives of the incoming flows:
\begin{eqnarray}\label{eq:6.11}
\big(\partial_yu, \partial_y\rho\big)(x,0+)=\big(\partial_y u, \partial_y \rho\big)(x,1-)=(0,0).
\end{eqnarray}}
In addition, we impose the compatibility condition on the second order derivatives of the incoming flows such that
{\begin{eqnarray}\label{eq:6.12}
\partial^2_yv_{0}(0+)=\partial^2_yv_{0}(1-)=0.
\end{eqnarray}}

\par Now, we can define the following periodic extension of the incoming flow.
\begin{eqnarray}\label{eq:6.13}
\begin{aligned}
\big(\check{u}_{0},\check{v}_{0}, \check{\rho}_{0}\big)(y)=
\left\{
\begin{array}{llll}
\big(u_{0},v_{0}, \rho_{0}\big)(y), \quad &y\in [0,1],\\
\big(u_{0},-v_{0}, \rho_{0}\big)(-y), \quad &y\in [-1,0],\\
\big(u_{0},v_{0}, \rho_{0}\big)(y-2j),\quad &y\in [2j-1,2j+1],
\end{array}
\right.
\end{aligned}
\end{eqnarray}
for $j=\pm1, \pm2,\cdots$.
Then $\big(\check{u}_{0},\check{v}_{0}, \check{\rho}_{0}\big)(y)\in C^{2}(\mathbb{R})$. 

{We will show the blow-up of solutions in general in a semi-infinitely long nozzle by a contradiction argument, if they satisfy
	\begin{eqnarray}\label{eq:6.15}
	\begin{split}
	\|(u,v,\rho)-(\underline{u}, 0, \underline{\rho})\|_{C^{1}([0,+\infty]\times[0,1])}\leq \varepsilon,
	\end{split}
	\end{eqnarray}
	where $\underline{u}, \underline{\rho}$ are the constant background state. More precisely, we suppose that the solutions to the initial-boundary value problem \eqref{eq:6.1}-\eqref{eq:6.12} exist globally. Then for any $(x,y)\in [0,+\infty]\times\mathbb{R}$, define
\begin{eqnarray}
\begin{aligned}
\big(u,v, \rho\big)(x,y)=
\left\{
\begin{array}{llll}
\big(u,v, \rho\big)(x,y), \quad &y\in [0,1],\\
\big(u,-v, \rho\big)(x,-y), \quad &y\in [-1,0],\\
\big(u,v, \rho\big)(x,y-2j),\quad &y\in [2j-1,2j+1],
\end{array}
\right.
\end{aligned}
\end{eqnarray}
for $j=\pm1, \pm2,\cdots$. Obviously, by \eqref{eq:6.10} and \eqref{eq:6.11}, $\big(u,v,\rho\big)(x,y)\in C^{1,1}([0,+\infty]\times\mathbb{R})$, and satisfies \eqref{eq:6.15} in $[0,+\infty]\times\mathbb{R}$.}

The system \eqref{eq:6.1}-\eqref{eq:6.2} is hyperbolic for $u>c$ and has two
eigenvalues
\begin{eqnarray}\label{eq:6.3}
\lambda_{-}=\frac{uv-c\sqrt{u^{2}+v^{2}-c^{2}}}{u^{2}-c^{2}},\quad
\lambda_{+}=\frac{uv+c\sqrt{u^{2}+v^{2}-c^{2}}}{u^{2}-c^{2}},
\end{eqnarray}
which are genuinely nonlinear since
\begin{eqnarray}\label{eq:6.7}
\partial_{Z_{+}}\lambda_{-}>0,\quad \partial_{Z_{-}}\lambda_{+}>0.
\end{eqnarray}
Here $Z_{\pm}$ are the Riemann invariants of system \eqref{eq:6.1}-\eqref{eq:6.2} with
\begin{eqnarray}\label{eq:6.6}
Z_{+}=\theta+\Theta(q),\quad Z_{-}=\theta-\Theta(q),
\end{eqnarray}
where $\Theta(q)=\int^{q}\frac{\sqrt{\tau^{2}-c^{2}(\tau)}}{\tau c(\tau)}d\tau$,
and
$\theta=\arctan\Big(\frac{v}{u}\Big),\ q=\sqrt{u^{2}+v^{2}}.$
Moreover,
\begin{eqnarray}\label{eq:6.8}
\partial_{+}Z_{-}=0,\quad \partial_{-}Z_{+}=0,
\end{eqnarray}
where $\partial_{\pm}:=\partial_{x}+\lambda_{\pm}\partial_{y}$.
Hence, {the global existence of solutions of the initial boundary value problem \eqref{eq:6.1}--\eqref{eq:6.10} yields the global existence of} the following Cauchy problem:
\begin{eqnarray}\label{eq:6.14}
\begin{aligned}
(\check{\mathbf{P}})\qquad \left\{
\begin{array}{llll}
\partial_{-}Z_{+}=0, \qquad  &(x,y)\in \mathbb{R}_{+}\times \mathbb{R},\\
\partial_{+}Z_{-}=0, \qquad  &(x,y)\in \mathbb{R}_{+}\times \mathbb{R},\\
Z_{+}(0,y)=\check{Z}_{+,0},   \qquad  &y\in \mathbb{R},\\
Z_{-}(0,y)=\check{Z}_{-,0}, \qquad   &y\in \mathbb{R},
\end{array}
\right.
\end{aligned}
\end{eqnarray}
where $\check{Z}_{+,0}, \check{Z}_{-,0}$ are functions of  $\check{u}_{0}(y),\check{v}_{0}(y), \check{\rho}_{0}(y)$
{and
$Z_{\pm}$ satisfy
\begin{eqnarray}\label{eq:6.16}
\begin{split}
\|(Z_{+}, Z_{-})-(\underline{Z}_{+},\underline{Z}_{-})\|_{C^{1}([0,+\infty]\times[0,1])}\leq \mathcal{C}\varepsilon.
\end{split}
\end{eqnarray}
Here $\mathcal{C}>0$ and $\underline{Z}_{\pm}$ depend only on $\underline{u}$ and $\underline{\rho}$.}

{For a given $T>0$, suppose that $(Z_{+, 1}, Z_{-,1})$ and $(Z_{+, 2}, Z_{-,2})$ are two   solutions to the problem $(\check{\mathbf{P}})$ satisfying
the same initial data. Denote by $\delta Z_{+}=Z_{+, 1}-Z_{+, 2}$ and $\delta Z_{-}=Z_{-, 1}-Z_{-, 2}$, then $(\delta Z_{+}, \delta Z_{-})$ satisfies the following
\begin{eqnarray}\label{eq:6.17}
\begin{aligned}
\left\{
\begin{array}{llll}
\partial_{-}\delta Z_{+}=\big(\lambda_{-}(Z_{+,1}, Z_{-,1})-\lambda_{-}(Z_{+,2}, Z_{-,2})\big)\partial_{y}Z_{+,2}, \quad  &(x,y)\in \mathbb{R}_{+}\times \mathbb{R},\\
\partial_{+}\delta Z_{-}=\big(\lambda_{+}(Z_{+,1}, Z_{-,1})-\lambda_{+}(Z_{+,2}, Z_{-,2})\big)\partial_{y}Z_{-,2}, \quad  &(x,y)\in \mathbb{R}_{+}\times \mathbb{R},\\
\delta Z_{+}(0,y)=0,   \qquad  &y\in \mathbb{R},\\
\delta Z_{-}(0,y)=0, \qquad   &y\in \mathbb{R},
\end{array}
\right.
\end{aligned}
\end{eqnarray}
So, we can employ the characteristic methods to get that
\begin{eqnarray}\label{eq:6.16}
\begin{split}
\|(\delta Z_{+},\delta Z_{-})\|_{C^{0}(\mathbb{R}_{+}\times \mathbb{R}}
&\leq \mathcal{C}T\|(\partial_{y} Z_{+,2},\partial_{y}Z_{-,2})\|_{C^{0}(\mathbb{R}_{+}\times \mathbb{R}}
\|(\delta Z_{+},\delta Z_{-})\|_{C^{0}(\mathbb{R}_{+}\times \mathbb{R})}\\
& \leq \mathcal{C}\varepsilon\|(\delta Z_{+},\delta Z_{-})\|_{C^{0}(\mathbb{R}_{+}\times \mathbb{R})},
\end{split}
\end{eqnarray}
where $\mathcal{C}$ depends on $T$ and $\underline{u}, \underline{\rho}$.
This implies that $(Z_{+, 1}, Z_{-,1})=(Z_{+, 2}, Z_{-,2})$ for any $(x,y)\in [0,T]\times \mathbb{R}$, \emph{i.e.},
the solutions to the problem $(\check{\mathbf{P}})$ is unique in $C^1(\mathbb{R}_+\times\mathbb{R})$.}

\par On the other hand, notice the fact that {$Z_{\pm,0}(-1)=Z_{\pm,0}(1)$}. Thus,
 if $\check{u}_{0}(y),\check{v}_{0}(y), \check{\rho}_{0}(y)$ are not constants, then there is a point in $[-1,-1]$ such that $\partial_{y}\check{Z}_{+,0}<0$ or
$\partial_{y}\check{Z}_{-,0}<0$. Then, following the result of P. Lax in \cite{pl}, {we obtain the contradiction since the $C^1$ solution of the Cauchy problem \eqref{eq:6.14} will blow up.
Therefore,} we have the following theorem:

\begin{theorem}\label{thm:6.1}
For $\check{q}_{0}>\check{c}_{0}$ and $\gamma>1$, if the incoming flow is not a constant and satisfies \eqref{eq:6.12}, then $\partial_{y}Z_{+}$ or $\partial_{y}Z_{-}$ will
becomes infinity in a finite time, $i.e.$, the solutions to the problem $(\check{\mathbf{P}})$
will blow-up in a finite time.
\end{theorem}


\bigskip

\section*{Acknowledgements}
F. Huang's research was supported in part  by NSFC Grant No. 11371349 and Key Research Program of Frontier Sciences, CAS.
D. Wang was supported in part by NSF grants DMS-1312800 and DMS-1613213.
W. Xiang was supported in part by the Research Grants Council of the HKSAR, China (Project No. CityU 21305215, Project No. CityU 11332916 and Project No. CityU 11304817).

\end{document}